\documentclass{amsart}
\usepackage{amsthm,amssymb,latexsym,amsmath,color,mathrsfs,caption}
\usepackage[all]{xy}
\usepackage[utf8]{inputenc}
\usepackage{multirow}
\usepackage{makecell}
\usepackage{marginnote}
\usepackage{float}
\usepackage{aliascnt}
\usepackage{hyperref}
\hypersetup{
	colorlinks=true,
	linkcolor=blue,
	filecolor=magenta, 
	urlcolor=cyan,
}
\newcommand{\p}[1]{{\mathbb{P}^{#1}}}
\newcommand{\pn}{{\mathbb{P}^n}}

\newcommand{\op}[1]{{\mathcal O}_{\mathbb{P}^{#1}}}
\newcommand{\opn}{{\mathcal O}_{\mathbb{P}^n}}

\newcommand{\tp}[1]{{\rm T}{\mathbb{P}^{#1}}}
\newcommand{\tpn}{{\rm T}{\mathbb{P}^{n}}}

\newcommand{\sing}{\operatorname{Sing}}
\newcommand{\supp}{\operatorname{Supp}}

\newcommand{\calb}{{\mathcal B}}

\newcommand{\cald}{{\mathcal D}}

\newcommand{\calf}{{\mathcal F}}

\newcommand{\calh}{{\mathcal H}}
\newcommand{\cali}{{\mathcal I}}
\newcommand{\calj}{{\mathcal J}}

\newcommand{\calr}{{\mathcal R}}

\newcommand{\inext}{{\mathcal E}{\it xt}}

\newcommand{\Ext}{\operatorname{Ext}}
\newcommand{\Hom}{\operatorname{Hom}}

\DeclareMathOperator{\coker}{coker}
\DeclareMathOperator{\im}{im}

\DeclareMathOperator{\rk}{{rk}}

\def\sD{{\mathscr{D}}}
\def\sF{{\mathscr{F}}}
\def\sG{{\mathscr{G}}}

\newcommand{\into}{\hookrightarrow}
\newcommand{\onto}{\twoheadrightarrow}

\newtheorem{theorem}{Theorem}
\newtheorem*{mthm}{Main Theorem}
\newtheorem{proposition}[theorem]{Proposition}

\newtheorem{lemma}[theorem]{Lemma}
\newtheorem{corollary}[theorem]{Corollary}

\theoremstyle{definition}
\newtheorem{remark}[theorem]{Remark}
\newtheorem{example}[theorem]{Example}

\definecolor{grn}{rgb}{0, 0.50, 0.0}

\makeatletter
\@namedef{subjclassname@2020}{%
	\textup{2020} Mathematics Subject Classification}
\makeatother

\title[Codimension one distributions of degree 2 on $\p3$]{Codimension one distributions of degree $2$ on the three-dimensional projective space}

\usepackage[foot]{amsaddr}
\author[H. Galeano]{Hugo Galeano}
\email{hagaleano@correo.unicordoba.edu.co}

\address[HG]{Universidad de Córdoba \\
	Carrera 6 No. 77-305 \\ Montería - Córdoba, Colombia}

\author[M. Jardim]{Marcos Jardim}
\email{jardim@unicamp.br}

\address[MJ]{IMECC - UNICAMP \\ Departamento de Matem\'atica \\
	Rua S\'ergio Buarque de Holanda, 651\\ 13083-970 Campinas-SP, Brazil}

\author[A. Muniz]{Alan Muniz}
\address[AM]{São Gonçalo-RJ, Brazil}
\email[A. Muniz]{alannmuniz@gmail.com}



\begin{document}

	\begin{abstract}
		We establish a full classification of degree $2$ codimension one distributions on $\mathbb{P}^3$ according to invariants of their tangent sheaves.
	\end{abstract}
	\maketitle
	
	\tableofcontents
	\section{Introduction}
	
	Let $\tpn$ and $\Omega^1_{\pn}$ denote the tangent and cotangent bundles of the complex projective space $\pn$. A \emph{codimension one holomorphic distribution} $\sD$ on $\pn$ is the data of a closed subset $|Z|$ and, for each point $p\in \pn\setminus |Z|$, a subspace $T_\sD(p) \subset {\rm T}_p\pn$ of codimension one. For every such distribution $\sD$ there exists a twisted $1$-form $\omega\in H^0(\Omega^1_{\pn}(d+2))$ for some $d\ge0$ such that $|Z|$ is its vanishing set and the fiber $T_\sD(p)$ at the point $p$ is precisely $\ker \omega(p)$. 
	
	In sheaf theoretical terms, $\omega$ defines a morphism $\omega \colon \tpn \rightarrow \opn(d+2)$ whose image is a twisted sheaf of ideals $\mathscr{I}_Z(d+2)$ defining a natural scheme structure to $|Z|$; we call $Z=:\sing(\sD)$ the \emph{singular scheme} of $\sD$. It is assumed that $\sing(\sD)$ has codimension at least $2$ and we denote by $\sing_{n-2}(\sD)$ the union of its irreducible components of pure codimension $2$. The sheaf $T_\sD\simeq\ker\omega$, called the \emph{tangent sheaf} of $\sD$, is a saturated subsheaf of $\tpn$ of rank $n-1$; in particular it is reflexive. Moreover, the integer $d$ is called the \emph{degree} of $\sD$.

	The information provided in the previous paragraph can be neatly described in terms of a short exact sequence as follows: 
	\[
	\sD ~:~ 0 \longrightarrow T_\sD \longrightarrow \tpn \stackrel{\omega}{\longrightarrow} \mathscr{I}_Z(d+2) \longrightarrow 0.
	\]
	
	A codimension one distribution $\sD$ is \emph{integrable}, i.e. it defines a \emph{foliation}, if $T_\sD \subset \tpn$ is closed under the Lie bracket, $[T_\sD,T_\sD]\subset T_\sD$; this is equivalent to the condition $\omega\wedge d\omega=0$ on the 1-form $\omega$. 
	
	The study of distributions, and especially foliations, on $\pn$ from the point of view of Algebraic Geometry has been a very active theme of research in the past few decades. A central problem has been the classification of codimension one foliations on $\pn$ with a given degree $d$. In \cite{Jou}, Jouanolou classified codimension one foliations of degrees $0$ and $1$ on $\pn$; in \cite{CLN}, Cerveau and Lins Neto classified codimension one foliations of degree $2$ on $\pn$. These authors describe the irreducible components of the algebraic set
	$$ \calf(d,n):= \{[\omega]\in \mathbb{P}H^0(\Omega^1_{\pn}(d+2)) \mid \omega\wedge d\omega=0 \} $$
	for $d=0,1$ and $2$, which can be regarded as parametrizing codimension one foliations of degree $d$ on $\pn$; their classification is given in terms of a naive deformation theory for the 1-form $\omega$.

	If we remove the integrability condition then any two $1$-forms in $H^0(\Omega^1_{\pn}(d+2))$ can be deformed into one another with no regard of the distributions they define. Thus a finer moduli theory was needed to study such objects.
	A novel approach to the study of flat deformations of codimension one distributions, regardless of integrability, was introduced in \cite{CCJ1}. The authors propose a classification in terms of natural topological invariants associated to a codimension one distribution, namely the Chern classes of tangent sheaf $T_\sD$, and provide a full description for $T_\sD$ and $Z$ when $d=0$ or $1$, see \cite[Proposition 7.1]{CCJ1} and \cite[Section 8]{CCJ1}, respectively. In addition, a classification of codimension one distribution on $\p3$ of degree 2 with locally free tangent sheaf and reduced singular scheme is also given, see \cite[Theorem 9.5]{CCJ1}.
	
	The aim of this work is to present a complete picture for codimension one distributions on $\p3$ such that $d=2$ or, equivalently, $c_1(T_\sD) = 0$. More precisely, we describe all possible tangent sheaves, addressing (semi)stability, Chern classes and spectra, if applicable. In addition, we also describe all possible singular schemes. Our main result is the following.
	
	\begin{mthm}\label{maintheorem}
		Let $\sD$ be a codimension one distribution of degree 2 on $\p3$. Then $T_\sD$ is $\mu$-semistable whenever it does not split as a sum of line bundles; it can be strictly $\mu$-semistable only when $(c_2(T_\sD),c_3(T_\sD))=(1,2)$ or $(2,4)$. In addition, the second and third Chern classes and spectra of $T_\sD$ are listed in Table \ref{deg 2 table}, where $\sing_1(\sD)$ is given. Finally, $\sing(\sD)$ is never contained in a quadric surface.
	\end{mthm}

	\begin{table}[t] \centering
		\begin{tabular}{|c|c|c|c|}
			\hline
			$c_2(T_\sD)$ & $c_3(T_\sD)$ & Spectrum & $\sing_1(\sD)$    \\ \hline \hline
			6     & 20  & \{-3,-2,-1,-1,-1\} & empty       \\ \hline
			5     & 14  & \{-2,-2,-1,-1,-1\}  & line    \\ \hline
			\multirow{3}{*}{4} & 10  & \{-2,-1,-1,-1\} & conic   \\ \cline{2-4} 
			& 8  & \{-1,-1,-1,-1\} & two skew lines \\ \cline{2-4} 
			& 6 & \{-1,-1,-1,0\} & double line of genus $-2$\\ \hline %
			\multirow{5}{*}{3} & 8  & \{-2,-1,-1\} & plane cubic curve     \\ \cline{2-4} 
			& 6  & \{-1,-1,-1\} & twisted cubic \\ \cline{2-4} 
			& 4  & \{-1,-1,0\} & conic $\sqcup$ line \\ \cline{2-4} 
			& 2  & \{-1,0,0\} & three skew lines \\ \cline{2-4} 
			& 0   & \{0,0,0\} & double line of genus $-2$ $\sqcup$ line \\ \hline
			\multirow{3}{*}{2} & 4   & \{-1,-1\} & elliptic quartic curve\\ \cline{2-4} 
			& 2   & \{-1,0\} & rational quartic curve\\ \cline{2-4} 
			& 0   & \{0,0\} & twisted cubic $\sqcup$ line\\ \hline
			\multirow{2}{*}{1} & 2   & \{-1\} & curve of degree 5, genus 2   \\ \cline{2-4} 
			& 0   & \{0\} & elliptic curve of degree 5   \\ \hline
			0     & 0   & split & ACM curve of degree 6 genus 3 \\ \hline
			-1     & 0   & split & ACM curve of degree 7 genus 5 \\ \hline
		\end{tabular}
		\caption{Possible Chern classes, spectra and singular loci of codimension one distribution of degree 2; the last column describes a generic point in the irreducible component of the Hilbert scheme that contains $\sing_1(\sD)$. 
		} 
		\label{deg 2 table}
	\end{table}
	
	We recall that the spectrum of a rank 2 reflexive sheaf is a multiset of integers that encodes partial information on its cohomology, see Section \ref{subsect:spectrum}.

	The bulk of this article is dedicated to the proof of \hyperref[maintheorem]{Main Theorem}. Some cases, namely those with $c_2(T_\sD)\le1$ and $(c_2(T_\sD),c_3(T_\sD))=(2,4)$ or $(3,8)$, where already established in \cite{CCJ1}; in addition, the case $(c_2(T_\sD),c_3(T_\sD))=(6,20)$ corresponds to 1-forms with only isolated singularities, which are known to form an open subset of $H^0(\Omega^1_{\p3}(4))$. Furthermore, the following facts were already observed in \cite[Section 8]{CCJ1}:
	\begin{enumerate}
		\item if $c_2(T_\sD)=5$, then $T_\sD$ is $\mu$-stable and $c_3(T_\sD)=14$;
		\item if $c_2(T_\sD)=4$, then $T_\sD$ is $\mu$-stable and $0\le c_3(T_\sD)\le 10$;
		\item if $c_2(T_\sD)=3$, then $T_\sD$ is $\mu$-semistable and $0\le c_3(T_\sD)\le 8$;
		\item if $c_2(T_\sD)=2$, then $T_\sD$ is $\mu$-semistable and $0\le c_3(T_\sD)\le 4$.
	\end{enumerate}
	A precise description of the cases for which $T_\sD$ is either strictly $\mu$-semistable or not $\mu$-semistable is given in Theorem \ref{thm:stability} below. In fact, it is worth emphasising that we find two families of distributions with $(c_2(T_\sD),c_3(T_\sD))=(2,4)$, one with $T_\sD$ $\mu$-stable and the other with $T_\sD$ strictly $\mu$-semistable, see Section \ref{sect:c2=2}.
	
	For the sake of completeness, we begin our proof by revising the basic properties of codimension one distributions of degree 2 mentioned above (see Section \ref{sect:pre}), as well as the cases with $c_2(T_\sD)\le1$ in Section \ref{sect:c2<=1}.
	
	The remaining cases are established using three different techniques, introduced in Section \ref{sect:basics}. The key point is that once $(c_2(T_\sD),c_3(T_\sD))$ is fixed, the degree and arithmetic genus of $\sing_1(\sD)$ can be computed via the expressions in display \eqref{eq:chern dist from folbc deg2}. When the degree is less than 5, the curves with these given degree and arithmetic genus are explicit described in the literature, especially \cite{N,NS:deg4}. So if we construct a codimension one distribution $\sD$ with a given $\sing_1(\sD)$, then $T_\sD$ will have the desired properties.
	
	The first technique, described in Section \ref{subsect:dist from folbc}, relies on codimension two distributions on $\p3$, i.e., \emph{foliations by curves}: given a codimension two distribution, one can find a codimension one distribution containing it and whose singular locus can be explicitly described, see Proposition \ref{prop: const dist folbc}. This result is applied to prove the existence of codimension one distributions $\sD$ of degree 2 such that $\sing_1(\sD)$ is a pair of disjoint lines, a smooth conic or a line, see Sections \ref{sec:4-8}, \ref{sec:4-10} and \ref{sect:c2=5}, establishing the existence of the cases $(c_2(T_\sD),c_3(T_\sD))=(4,8),(4,10)$ and $(5,14)$, respectively. It is also interesting to note that our results can be regarded as a proof of existence of stable rank 2 reflexives sheaves with the spectra given in Table \ref{deg 2 table}.
	
	The second approach, described in Section \ref{subsect:syzygies}, relies on constructing an explicit expression for a 1-form $\omega$ vanishing along a given curve. This procedure provides explicit examples for the cases $(c_2(T_\sD),c_3(T_\sD))=(3,2) , (3,4), (3,6)$ and $(4,6)$ (see Sections \ref{sect:c2=3 2-6} and \ref{sect:c2=4}) and also allows to show, as a consequence of Corollary \ref{cor:2line}, that the cases $(c_2(T_\sD),c_3(T_\sD))=(4,0),(4,2)$ and $(4,4)$ cannot be realized. All explicit examples presented in this paper were computed with Macaulay2 \cite{M2}.
	
	Finally, the third technique looks into properties of stable rank 2 reflexive sheaves with desired second and third Chern classes. Exploring these properties, one can prove the existence of a monomorphism from a given sheaf $F$ into $\tp3$ whose cokernel is torsion free. This is applied to the case $(c_2(T_\sD),c_3(T_\sD))=(3,0)$, for which a profound knowledge of stable rank 2 reflexive sheaves is available. The advantage of this more delicate technique is that it allows for the construction not only of a particular example, but actually for the construction of the whole family of distributions with the prescribed invariants, see Theorem \ref{thm:instanton3} and Proposition \ref{prop:moduli30}.
	
	For the sake of comparison, we recall in Table \ref{deg 2 foliations} the Cerveau-Lins Neto classification of codimension one foliations of degree 2, comparing with classification scheme outlined in \cite{CCJ1} and further expanded here. Any integrable distribution of degree $2$ must be in the closure of one of those six families: $R(a,b)$ and $L(d_1,..,d_4)$ are the families of \emph{rational} and \emph{logarithmic foliations}; $E(3)$ is the family of \emph{exceptional foliations}; $S(2,3)$ are \emph{pullback foliations}. 
	
	Our thorough study of codimension one distributions of degree 2 on $\p3$, together with previous study of distributions of degrees 0 and 1, allows us to propose two questions. Let $\sD$ a degree $d\geq 0$ distribution  on $\p3$. 
	\begin{enumerate}
		\item Is is true that the singular scheme $\sing(\sD)$ is never contained in a hypersurface of degree $\leq d$? This was already remarked in \cite{CJMdeg1} and we see here that it holds for $d=2$.
		\item If $\sD$ is integrable then is it true that
		\[
		c_2(T_\sD) \leq d^2-d+1?
		\] 
		If we assume that $\sing_1(\sD)$ is reduced, then this bound is an easy consequence of \cite[Proposição 2.6]{FassarellaThesis}. Note also that the equality is satisfied for generic rational foliations $\sD$ of type $R(1,d)$. 
	\end{enumerate}

	\begin{table}[t] \centering
		\begin{tabular}{|c | c | c | c |} \hline
			Foliation & $c_2(T_\sD)$ & $c_3(T_\sD)$ & $T_\sD$  \\ \hline
			$R(1,3)$  & 3  & 8  & stable     \\ \hline
			$R(2,2)$  & 2  & 4  & stable     \\ \hline
			$L(1,1,2)$ & 1  & 2  & str. $\mu$-semistable \\ \hline
			$L(1,1,1,1)$ & 0  & 0  & $\op3\oplus\op3$  \\ \hline
			$E(3)$  & 0  & 0  & $\op3\oplus\op3$ \\ \hline
			$S(2,3)$  & -1  & 0  & $\op3(1)\oplus\op3(-1)$ \\ \hline
		\end{tabular}
		\caption{The first column lists the 6 irreducible components of $\calf(2,3)$ according to Cerveau and Lins Neto \cite{CLN}. Each line provides a description of the tangent sheaves for a generic point lying in each component.}
		\label{deg 2 foliations}
	\end{table}

	\subsection*{Acknowledgments}
	We thank Maurício Corrêa and Thiago Fassarella for enlightening discussions, comments and suggestions. We also thank Jorge Vitório Pereira for pointing out the relation between $\sing_1(\sD)$ and the Gauss map. HG was supported by a PhD grant from CAPES, and the results of his PhD thesis are present here. MJ is supported by the CNPQ grant number 302889/2018-3 and the FAPESP Thematic Project 2018/21391-1. The authors also acknowledge the financial support from Coordenação de Aperfeiçoamento de Pessoal de Nível Superior - Brasil (CAPES) - Finance Code 001.


	\section{Preliminaries and notation} \label{sect:pre}
	
	We begin by setting up the notation and nomenclature to be used in the rest of the paper.
	
	\subsection{Codimension one distributions on \texorpdfstring{$\p3$}{P3}}
	
	A codimension one distribution on $\p3$ is given by an exact sequence
	\begin{equation}\label{seq:dist-p3}
		\sD ~:~ 0 \longrightarrow T_\sD \longrightarrow \tp3 \stackrel{\omega}{\longrightarrow} \mathscr{I}_Z(d+2) \longrightarrow 0,
	\end{equation}
	where $d\ge0$ is called the \emph{degree} of $\sD$, and $Z$ is its \emph{singular scheme}. The last map is given by a global section $\omega \in H^0(\Omega_{\p3}^1(d+2))$ which can be seen, from the Euler sequence, as a homogeneous polynomial $1$-form  $\omega = \sum_{i=0}^{3}A_i dx_i$, on $\mathbb{C}^4$, such that $\deg(A_i) = d+1$ and the contraction with the Euler radial vector field $\sum_{i=0}^{3} x_i \frac{\partial}{\partial x_i} $ vanishes, i.e., $ \sum_{i=0}^{3}A_i x_i = 0$. Moreover, $Z$ is precisely the vanishing locus of $\omega$ with its natural scheme structure locally given by the $A_i$.
	
	Let $F = \gcd(A_0,A_1,A_2,A_3)$, Then $\omega = F\omega'$ for $\omega'$ a homogeneous $1$-form of lower degree. Since both  $\omega$ and $\omega'$, define the same distribution on the Zariski open subset $\p3 \setminus \{F=0\}$, we will assume throughout that $F=1$, i.e. the $A_i$ are relatively prime. This is equivalent to assuming that $\dim Z \leq 1$.
	
	In general, $Z$ may have both 0- and 1-dimensional components. Letting $U$ be the maximal 0-dimensional subsheaf of $\mathcal{O}_Z$, we obtain the exact sequence
	\begin{equation}\label{seq:sing decomp dist}
		0 \longrightarrow U \longrightarrow \mathcal{O}_Z \longrightarrow \mathcal{O}_C \longrightarrow 0,
	\end{equation}
	where $C$ is a pure 1-dimensional scheme, i.e., a Cohen-Macaulay curve. We also denote it by $\sing_1(\sD):= C$. One can show that \cite[Theorem 3.1]{CCJ1}:
	\begin{align}
		\begin{split} \label{eq:chern-deg,pa}
			c_1(T_\sD) & = 2 - d; \\
			c_2(T_\sD) & = d^2 + 2 - \deg(C); \\
			c_3(T_\sD) & = {\rm length}(U) = 
			d^3 + 2d^2 + 2d -\deg(C)\cdot(3d-2) + 2p_a(C) - 2, 
		\end{split}
	\end{align}
	where $p_a(C)$ denotes the arithmetic genus of $C$.
	
	Since this paper is dedicated to the case $d=2$, we specialize it further. Then $c_1(T_\sD)=0$ and $T_\sD\simeq T_\sD^*$, since $T_\sD$ is a rank two reflexive sheaf. The formulas in display \eqref{eq:chern-deg,pa} simplify to 
	\begin{align} \label{eq:chern sing1 deg2}
		\begin{split}
			c_2(T_\sD) &= 6 - \deg(C); \\
			c_3(T_\sD) &=  18 - 4\deg(C) + 2p_a(C).
		\end{split}
	\end{align}
	
	Since $U$ is supported in dimension zero, $\inext^j(U, \op3) = 0$ for $j \leq 2$, see \cite[Proposition 1.1.6]{HL}. Then \eqref{seq:sing decomp dist} implies that
	\[
	\inext^1( \mathscr{I}_Z(4) , \op3) = \inext^2( \mathcal{O}_Z , \op3(-4)) = \inext^2( \mathcal{O}_C , \op3(-4)) = \omega_C.
	\]
	Dualizing the sequence in display \eqref{seq:dist-p3} and using $d=2$, we obtain
	\begin{equation}\label{seq:zeta}
		0 \longrightarrow \op3(-4) \longrightarrow \Omega^1_{\p3} \longrightarrow T_\sD
		\stackrel{\zeta}{\longrightarrow} \omega_C  \longrightarrow 0.
	\end{equation}
	Computing the cohomology we get the following key lemma; a similar result holds in any given degree. 
	
	\begin{lemma}\label{lem:h2=0}
		Let $\sD$ be a degree $2$ codimension one distribution on $\mathbb{P}^3$ then 
		\begin{enumerate}
			\item $h^0(T_\sD(1)) = h^0(\omega_C(1))$;
			\item $h^2(T_\sD)\leq 1$;
		\end{enumerate}
		and for $p\geq 1$ we have
		\begin{enumerate}
			\setcounter{enumi}{2}
			\item $h^1(T_\sD(p)) = h^1(\omega_C(p))$;
			\item $h^2(T_\sD(p))=0$.
		\end{enumerate}
	\end{lemma}
	
	Note that since $C$ is a Cohen-Macaulay curve, Serre duality holds and we may use $h^i(\omega_C(p)) = h^{1-i}(\mathcal{O}_C(-p))$ for $i\in {0,1}$ and every $p\in \mathbb{Z}$, see \cite[III Corollary 7.7]{HART:AG} . 
	
	An important invariant of $C$ is its Rao module $M_C = \bigoplus_p H^1(\mathscr{I}_C(p))$. From the sequence \eqref{seq:dist-p3} and using that $U = \mathscr{I}_C /\mathscr{I}_Z$ is zero-dimensional, we get
	\[
	h^1(\mathscr{I}_C (p)) \leq h^2(T_\sD(p-4)) \text{ for } p\in \mathbb{Z}. 
	\]
	Thus we have a restriction on the curves that can appear in the singular loci of degree two distributions.
	
	
	\subsection{The spectrum of a rank 2 reflexive sheaf}\label{subsect:spectrum}
	
	For a normalized rank two reflexive sheaf $F$ on $\p3$, such that $h^0(F(-1)) = 0$, there exists a unique multiset (i.e. repeated elements are allowed) of integers $\{k_1, \dots, k_{c_2(F)}\}$, called the \emph{spectrum} of $F$, that encodes partial information on the cohomology of $F$. It was first defined by Barth and Elencwajg for locally free sheaves and later extended for reflexive sheaves by Hartshorne; we will use \cite[Section 7]{H2} as reference. 
	
	\begin{remark}\label{rem:properties of spectrum}
		Let $F$ be a $\mu$-semistable rank 2 reflexive sheaf on $\p3$ with $c_1(F)=0$ and set $c_2(F)=n$ and $c_3(F)=2l$. If $\{k_1, \dots, k_n\}$ is the spectrum of $F$ then:
		\begin{align}\label{eq:def spectrum}
			\begin{split}
				h^1(F(p)) & = \sum_{i=1}^n h^0(\p1,\op1(k_i+p+1)) \text{ for each } p\leq -1;\\
				h^2(F(p)) &= \sum_{i=1}^n h^1(\p1,\op1(k_i+p+1)) \text{ for each } p\geq -3.
			\end{split}
		\end{align}
		The possible spectra of a sheaf with given Chern classes can be determined via the following properties, see \cite[Propositions 7.2 and 7.3 and Theorem 7.5]{H2}.
		\begin{enumerate}
			\item $\sum_{i=1}^n k_i = - l$;
			\item if $k>0$ belongs to the spectrum, then so do $1,\dots,k$; if, in addition, $F$ is $\mu$-stable then $0$ must also occur;
			\item if $k<0$ belongs to the spectrum, then so do $k,k+1,\dots,-1$; 
			\item if $F$ is $\mu$-stable then either $0$ occurs or $-1$ occurs at least twice;
			\item If $F$ is locally free, then the spectrum is symmetric, i.e., if $k_i$ occurs, then so does $-k_i$.
		\end{enumerate}
		
	\end{remark}
	
	As we have seen in Lemma \ref{lem:h2=0}, being the tangent sheaf of a codimension one distribution imposes certain restrictions on $h^1(F(p))$ and $h^2(F(p))$, which can in turn be used to rule out several possible spectra for sheaves with given Chern classes. Studying the spectrum of tangent sheaves also leads to a precise knowledge of the associated singular schemes, and vice-versa, allowing us to collect the information displayed in Table \ref{deg 2 table}. Here is an immediate application.

	\begin{lemma}\label{lem:bound spectrum}
		Let $\sD$ be a codimension one distribution on $\mathbb{P}^3$ of degree $2$ and let $\{k_1, \dots, k_n\}$ be the spectrum of $T_\sD$. Then $k_i \geq -3$ for every $i$ and $-3$ occurs at most once. Moreover, $1$ occurs in the spectrum if and only if $\sing(\sD)$ is contained in a quadric surface. 
	\end{lemma}
	
	\begin{proof}
		From Lemma \ref{lem:h2=0} we get $h^2(T_\sD) \leq 1$; then from \eqref{eq:def spectrum} we must have $k_i \geq -3$ for every $i$ and $-3$ occurs at most once. Computing the cohomology for the sequence \eqref{seq:dist-p3} we see that $h^1(T_\sD(-2)) \neq 0 $ if and only if $\sing(\sD)$ is contained in a surface of degree at most $2$. On the other hand, $h^1(T_\sD(-2)) \neq 0 $ if and only if $1$ occurs in the spectrum.
	\end{proof}
	
	In Sections \ref{sect:c2<=1} through \ref{sect:c2=6} we will use Remark \ref{rem:properties of spectrum} to compute the possible spectra for tangent sheaves of distributions of degree 2; they are listed in the third column of Table \ref{deg 2 table}. In particular, we check that $1$ never appears in any spectrum, thus the second part of Lemma \ref{lem:bound spectrum} guarantees that $\sing(\sD)$ is never contained in a quadric surface, as stated in the \hyperref[maintheorem]{Main Theorem}.

	
	\subsection{Foliations by curves on \texorpdfstring{$\p3$}{P3}}
	
	Consider now codimension $2$ distributions on $\p3$; they are given by an exact sequence
	\begin{equation} \label{fbc1}
		\sG ~:~ 0 \longrightarrow \op3(1-k) \stackrel{\nu}{\longrightarrow} \tp3 \longrightarrow N_{\sG} \longrightarrow 0;
	\end{equation}
	where $N_{\sG}$, called the normal sheaf, is torsion free. The integer $k\ge0$ is called the \emph{degree} of the distribution determined by the vector field $\nu\in H^0(\tp3(k-1))$, up to scalar factor. Since every one-dimensional distribution is automatically integrable, we call \eqref{fbc1} a \emph{foliation by curves of degree $k$} on $\p3$. For a detailed account on foliations by curves we refer to \cite{CaJS, CJM}. We will be mostly interested in the cases $k=0$ and $k=1$.
	
	In order to make the role of the singular locus more explicit, we dualize the sequence in display \eqref{fbc1} to obtain
	\begin{equation}\label{fol curves deg 1}
		0 \longrightarrow N_\sG^* \longrightarrow \Omega^1_{\p3} \stackrel{\iota_\nu}{\longrightarrow} \mathscr{I}_W(k-1) \longrightarrow 0 ,
	\end{equation}
	where $W$ is the scheme of zeros of the vector field $\nu$. From the Euler sequence, $\nu$ is an equivalence class of homogeneous polynomial vector fields $\widetilde{\nu} = \sum_{i=0}^3 F_i\frac{\partial}{\partial x_i}$ with $\deg(F_i) = k$, on $\mathbb{C}^4$, modulo multiples of the radial vector field, i.e., $\widetilde{\nu}_1$ and $\widetilde{\nu}_2$ represent $\nu$ if and only if $\widetilde{\nu}_1 - \widetilde{\nu}_2 = P\left(\sum_{i=0}^{3} x_i \frac{\partial}{\partial x_i} \right)$ for some polynomial $P$.  We can choose $\widetilde{\nu}$ such that ${\rm div}(\widetilde{\nu}) =\sum_{i=0}^3 \frac{\partial F_i}{\partial x_i} = 0$; this choice is unique. Given any representative $\widetilde{\nu}$ of $\nu$, one can show, by comparing with \eqref{fol curves deg 1}, that $W$ is the locus of points where $\widetilde{\nu}$ and $\sum_{i=0}^{3} x_i \frac{\partial}{\partial x_i}$ are parallel. More precisely, $W$ is given by the homogeneous ideal generated by the $2\times 2$ minors of the matrix
	\begin{equation}\label{eq:minors}
		\begin{bmatrix}
			F_0 & F_1 & F_2 & F_3 \\ x_0 & x_1 & x_2 & x_3
		\end{bmatrix}.
	\end{equation}
	
	As for codimension one distributions, $W$ may not be pure dimensional, so we have an exact sequence
	\begin{equation}\label{seq:sing decomp fol by curves}
		0 \longrightarrow R \longrightarrow \mathcal{O}_W \longrightarrow \mathcal{O}_Y \longrightarrow 0,
	\end{equation}
	where $R$ is a 0-dimensional sheaf and $Y$ is a curve; we set $Y:=\sing_1(\sG)$ and $R:=\sing_0(\sG)$. The foliation by curves $\sG$ is said to be generic if $\dim W=0$, that is, $\nu$ only vanishes at points. 
	
	When $k=0$, the picture is very simple. The foliation $\sG$ is given by a constant vector field $\nu$ vanishing at only one point, and $N_\sG$ is a stable rank 2 reflexive sheaf given by a sequence of the following form
	$$ 0 \longrightarrow \op3 \longrightarrow \op3(1)^{\oplus 3} \longrightarrow N_\sG \longrightarrow 0; $$
	note that $N_\sG^*\simeq N_\sG(-3)$; see also \cite[Section 5]{CJM}. 
	
	When $k=1$ we have that $\sG$ is given by a linear vector field $\nu = \sum_{i,j =0}^3 a_{ij} x_j\frac{\partial }{\partial x_i}$ that we can choose such that the matrix $A = (a_{ij})$ has vanishing trace; this choice of $\nu$ defining $\sG$ is unique up to scalar multiples. Thus $\sG$ can be described by the structure of the matrix $A$, and we can provide a complete classification of foliation by curves of degree one.
	
	First, recall that a torsion free sheaf $E$ on $\pn$ is $\mu$-semistable if for every saturated subsheaf $F\subset E$,
	\[
	\frac{c_1(F)}{\rk F} \leq \frac{c_1(E)}{\rk E}.
	\]
	the sheaf $E$ is $\mu$-stable if it is $\mu$-semistable and the inequalities are always strict. If $E$ is a rank 2 reflexive sheaf, then we only need to consider $F = \opn(l)$ for some $l\in\mathbb{Z}$.

	\begin{theorem}\label{thm:fbc deg 1}
		Let $\sG$ be a foliation by curves of degree 1 on $\p3$. Then 
		\begin{enumerate}
			\item $N_\sG^*$ is a $\mu$-stable rank 2 reflexive sheaf with Chern classes $c_2(N_\sG^*) = 6$ and $c_3(N_\sG^*) = 4$, and $W$ is a $0$-dimension scheme of length $4$;
			\item $N_\sG^*$ is a strictly $\mu$-semistable reflexive sheaf with Chern classes $c_2(N_\sG^*) = 5$ and $c_3(N_\sG^*) = 2$ and given by an extension
			\begin{equation}\label{O-N-line}
				0\longrightarrow \op3(-2) \longrightarrow N_\sG^* \longrightarrow \mathscr{I}_L(-2) \longrightarrow 0,
			\end{equation}
			where $L$ is a line. In this case, $W = L' \cup p$ where $L'$ is a line and $p$ is a $0$-dimensional scheme of length $2$, possibly embedded in $L'$. Moreover, $p\subset L$ and $W$ is nonplanar; in particular, $L\neq L'$;
			\item $N_\sG^*=\op3(-2)\oplus\op3(-2)$, in particular $c_2(N_\sG^*) = 4$ and $c_3(N_\sG^*) = 0$, and $W$ consists of either 2 skew lines or a double line of genus $-1$.
		\end{enumerate}
	\end{theorem}
	\begin{proof}
		Let $A$ be the matrix associated with a vector field $\nu\in H^0(\tp3)$, as described above. 
		We argue that $\sing(\sG)$ corresponds to the eigenspaces of $A$ counted with multiplicities. Indeed, the minors of the matrix \eqref{eq:minors}, for $\nu = \sum_{i,j =0}^3 a_{ij} x_j\frac{\partial }{\partial x_i}$, vanish if and only if $(x_0,x_1,x_2,x_3)^T$ is an eigenvector of $A$. Hence $W = \sing(\sG)$ is a scheme supported at the projectivized set of eigenvectors of $A$. We have three cases:
		\begin{enumerate}
			\item $A$ has only eigenspaces of dimension $1$, then $\dim W = 0$;
			\item $A$ has one eigenspace of dimension $2$, then $\sing_1(\sG)$ is supported on a line $L_1$; 
			\item $A$ has two eigenspaces of dimension $2$, then $W$ is supported on two skew lines $L_1 \sqcup L_2$.
		\end{enumerate}
		If $A$ has an eigenspace of dimension $3$, then the vector field vanishes along a plane, and therefore the $\sG$ is not saturated. Finally, $A$ has an eigenspace of dimension $4$ if and only if $A=0$ since it has vanishing trace. 
		
		Due to \cite[Theorem 4.1]{CJM}, we have that $c_1(N_\sG^*) = -4$, $c_2(N_\sG^*) = 6- \deg( Y)$ and $c_3(N_\sG^*) = 2 + 2p_a(Y) $ (or $c_3(N_\sG^*) = 4$ if $Y = \emptyset$). It follows from the sequence \eqref{fol curves deg 1} that $h^0(N_\sG^\ast(1)) = 0$, hence $N_\sG^\ast$ is $\mu$-semistable.

		In the first case $\dim W = 0$. Then dualizing \eqref{fol curves deg 1} we see that $h^0(N_\sG^\ast(2)) = 0$ hence $N_\sG^\ast$ is $\mu$-stable. We also have that $c_2(N_\sG^\ast) = 6$ and $c_3(N_\sG^\ast) = 4$.
		
		In the second case, first assume $\sing_1(\sG) = L_1$. Then it follows that $c_2(N_\sG^\ast) = 5$, $c_3(N_\sG^\ast) = 2$, and $W = L_1 \cup p $ with $p$ being a 0-dimensional subscheme of length $2$. Then \cite[Lemma 2.1]{Ch} applied to $N_\sG^\ast(2)$ shows that $N_\sG^\ast$ is strictly $\mu$-semistable and fits in the exact sequence \eqref{O-N-line}. In particular, $L$ must contain the scheme $p$. It also follows from \eqref{O-N-line} that $h^1(N_\sG^\ast(1)) = 0$, and then $\eqref{fol curves deg 1}$ implies that $W$ is not planar.

		Still in the second case, assume that $\sing_1(\sG)$ has a nonreduced structure along $L_1$. Checking the Jordan normal forms we see that it is only possible if, up to a linear change of coordinates,
		\[
		A  = \begin{bmatrix}
			0 &1 &0 &0 \\ 0 &0 &0 &0 \\ 0 &0 &0 &1 \\ 0 &0 &0 &0 \\
		\end{bmatrix}
		\]
		hence $W$ is given by the homogeneous ideal $(x_1^2,x_1x_3,x_3^3, x_1x_2 -x_0x_3 )$ and $W$ is a double line of genus $-1$. Then $c_2(N_\sG^\ast) = 4$ and $c_3(N_\sG^\ast) = 0$ and $N_\sG^\ast(2)$ is a $\mu$-semistable sheaf with trivial Chern classes, hence it is the trivial vector bundle and $N_\sG^\ast = \op3(-2) \oplus \op3(-2)$. 
		
		In the third case, as in the last paragraph, $c_2(N_\sG^\ast) = 4$ and $c_3(N_\sG^\ast) = 0$. Hence $N_\sG^\ast = \op3(-2) \oplus \op3(-2)$ as well.
	\end{proof}

	
	\subsection{Stability for distributions and sub-foliations}
	
	In order to describe codimension one distributions, we analyse the stability of the possible tangent sheaves. If $\sD$ is a degree $d$ codimension one distribution on $\p3$ such that $T_\sD$ is not $\mu$-stable, there exists a line sub-bundle $\mathcal{O}_\p3(l) \subset T_\sD$ such that $2l \geq 2-d$. On the other hand, $\mathcal{O}_\p3(l) \hookrightarrow T_\sD \hookrightarrow \tp3$ can only exist if $l \leq 1$. For $d=2$, we must have
	$l =0$ or $1$; the later only occurs if $T_\sD$ is not $\mu$-semistable. These line sub-bundles induce sub-distributions of codimension two; next, we will describe them.

	If $\sD$ is a codimension one distribution of degree two then $-1 \leq c_2(T_\sD) \leq 6$. Indeed, we know from \eqref{eq:chern sing1 deg2} that $c_2(T_\sD) = 6 - \deg (\sing_1(\sD))$ and, on the other hand, the restriction of $\sD$ to a general plane $H$ is singular at $\sing_1(\sD) \cap H$, whence $\deg( \sing_1(\sD))\leq 7$; see \cite[p. 28]{CCJ1}.

	If $T_\sD$ is not $\mu$-semistable then $h^0(T_\sD(-1)) \neq 0$ and, owing to \cite[Lemma 4.3]{CCJ1}, it must split as sum of line bundles and, as $T_\sD \hookrightarrow \tp3$, the only possibility is $T_\sD = \op3(1) \oplus \op3(-1)$. This covers the case $c_2(T_\sD) = -1$ and, conversely, $T_\sD$ is $\mu$-semistable if $c_2(T_\sD) \geq 0$. 
	
	If $c_2(T_\sD) = 0$ then $T_\sD = \op3 \oplus \op3$. If $c_2(T_\sD) = 1$ then $c_3(T_\sD)= 0$ or $2$, owing to \cite[Theorem 8.2]{H2}. For $c_3(T_\sD)= 0$ we have that $T_\sD$ must be a null correlation bundle, which is stable; for $c_3(T_\sD)= 1$ \cite[Lemma 2.1]{Ch} shows that $T_\sD$ is strictly $\mu$-semistable. 
	
	If $c_2(T_\sD)\ge4$, then it follows from \cite[Proposition 6.3]{CCJ1} that the tangent sheaf $T_\sD$ is stable. We will now complete this picture, and the first claim of \hyperref[maintheorem]{Main Theorem}, with the following result. 
	
	\begin{theorem}\label{thm:stability}
		Let $\sD$ be a codimension one distribution on $\p3$ of degree 2. The stability of $T_\sD$ is described as follows:
		\begin{enumerate}
			\item $T_\sD$ is not $\mu$-semistable if and only if $c_2(T_\sD) = -1$; if that is the case $T_\sD = \op3(1) \oplus \op3(-1)$;
			\item If $T_\sD$ is strictly $\mu$-semistable then $(c_2(T_\sD), c_3(T_\sD)) = (0,0)$, $(1,2)$ or $(2,4)$;
		\end{enumerate}
		For all the other cases $T_\sD$ is $\mu$-stable.
	\end{theorem}
	
	\begin{proof}The first item follows from the above discussion; we will prove now the second one. Consider a codimension 1 distribution of degree 2
		$$ 
		\sD ~:~ 0 \longrightarrow T_\sD \stackrel{\phi}{\longrightarrow}\tp3 \longrightarrow \mathscr{I}_Z(4) \longrightarrow 0;
		$$
		also from the previous discussion, we may assume $c_2(T_\sD)\geq 2$. If $\sD$ is not $\mu$-stable then $h^0(T_\sD) \neq 0$; so let $\sigma\in H^0(T_\sD)$ be a non trivial section, and let $S:=(\sigma)_0$. Due to $\mu$-semistability, $S$ is a curve of degree $\deg(S)=c_2(T_\sD) \geq 2$; otherwise we could factor out the divisoral part of $S$, yielding a section in $H^0(T_\sD(-1))$. Moreover, $\sigma$ induces a (sub-)foliation by curves of degree 1
		\[
		\sG ~:~ 0 \longrightarrow \op3 \stackrel{\phi\circ\sigma}{\longrightarrow}\tp3 \longrightarrow N_\sG \longrightarrow 0.
		\]
		Since $\im (\phi\circ \sigma)^\vee \subset \im \sigma^\vee$, we see that $S\subseteq Y =\sing_1(\sG)$; in particular, we have $\deg(Y)\ge2$. According to Theorem \ref{thm:fbc deg 1}, we must have $N_\sG ^*=\op3(-2)\oplus\op3(-2)$ and $\sing(\sG) = Y$ consists of two skew lines; otherwise $Y$ would have degree one. As $S$ has degree at least two, we conclude that $S=Y$. Therefore, $c_2(T_\sD) = 2$ and $c_3(T_\sD) = 2p_a(S) - 2 + 8 = 4$.
	\end{proof}
	
	One advantage of the $\mu$-semistability to our classification is to bound the third Chern class. For a $\mu$-stable rank two reflexive sheaf $F$ on $\mathbb{P}^3$ with $c_1(F) = 0$ we have, owing to \cite[Theorem 8.2]{H2}, that
	\begin{equation}\label{eq:bound c3}
		c_3(F) \leq c_2(F)^2 - c_2(F) + 2.
	\end{equation}
	Also recall \cite[Corollary 2.4]{H2} that $c_3 \equiv c_1 c_2 \pmod{2}$, hence $c_3(F)$ is always even. 
	

	\section{Basic constructions}\label{sect:basics}
	
	In the remainder of this paper we will provide the existence or non-existence of distributions according to the given Chern classes. To fulfill our purpose, we will establish in this section the key results needed for the construction of examples of codimension one distributions of degree 2; and some results to discard the impossible cases.
	
	
	\subsection{Distributions from foliations by curves}\label{subsect:dist from folbc}
	
	Let $\sG$ be a degree $k$ foliation by curves as in display \eqref{fol curves deg 1} and let $\sigma\in H^0(N_\sG^*(l))$ be a section, for some $l\in \mathbb{Z}$. We assume that $X := (\sigma)_0$, the vanishing locus of $\sigma$, has codimension $2$. We will use $\sigma$ to produce a codimension one distribution.
	
	Now consider the composed monomorphism $\rho \colon \op3(-l)\stackrel{\sigma}{\hookrightarrow} N_\sG^* \hookrightarrow \Omega^1_{\p3}$ and let $F := \coker \rho$. Using the Snake Lemma, we can see that $F$ fits into the following exact sequence
	\begin{equation}\label{seq:one}
		0 \longrightarrow \mathscr{I}_X(l-k-3) \longrightarrow F \longrightarrow \mathscr{I}_W(k-1) \longrightarrow 0,
	\end{equation} 
	where $W = \sing(\sG)$. In particular, since $X$ and $W$ have codimension at least two, $F$ is torsion free. Therefore, we have a codimension one distribution of degree $l-2$ given by the exact sequence
	\[
	\sD ~:~ 0 \longrightarrow F^* \longrightarrow \tp3 \longrightarrow \mathscr{I}_Z(l) \longrightarrow 0,
	\]
	called the \emph{codimension one distribution induced by the pair} $(\sG,\sigma)$. Note that $\inext^1(F,\op3)\simeq \mathcal{O}_Z(l)$; it is also clear that $X \subset Z$.
	
	Dualizing the exact sequence in display \eqref{seq:one} yields the long exact sequence
	\begin{multline}\label{long sqc}
		0\longrightarrow \op3(1-k) \longrightarrow F^* \longrightarrow \op3(k+3-l) \stackrel{\eta}{\longrightarrow} \omega_Y(5-k) \longrightarrow \\
		\longrightarrow \mathcal{O}_Z(l) \longrightarrow \omega_X(k-l+7) \longrightarrow \omega_R \longrightarrow 0. 
	\end{multline}
	where $Y := \sing_1(\sG)$ and $R := \sing_0(\sG)$, noting that 
	$$ \inext^1(\mathscr{I}_W(k-1),\op3)\simeq \omega_Y(5-k). $$
	In addition, the morphism $\eta$ in display \eqref{long sqc} can also be regarded as a section $\eta\in H^0(\omega_Y(l-2k+2))$.

	\begin{proposition}\label{prop: const dist folbc}
		Let $\sG$ be a foliation by curves with $Y=\sing_1(\sG)$, and take a non zero section $\sigma\in H^0(N_\sG^*(l))$ such that $(\sigma)_0$ is a curve; let $\sD$ be the codimension one distribution induced by $(\sG,\sigma)$, and let $\eta\in H^0(\omega_Y(l-2k+2))$ be as above. If $\dim\coker\eta=0$, then
		\[
		\sing_1(\sD)=(\sigma)_0 \text{ , } \sing_0(\sD)=(\eta)_0 \text{ , }
		\]
		\[
		\text{and}~~ 0 \longrightarrow \op3(1-k) \longrightarrow T_\sD \longrightarrow \mathscr{I}_W(k+3-l) \longrightarrow 0.
		\]
	\end{proposition}

	\begin{proof}
		We will use the same notation as above. Breaking the exact sequence in display \eqref{long sqc} into short exact sequences, we extract the following two:
		\begin{gather}
			0 \longrightarrow \coker\eta \longrightarrow \mathcal{O}_Z(l) \longrightarrow \mathcal{O}_{Z'}(l) \longrightarrow 0; \label{four} \\
			0\longrightarrow \mathcal{O}_{Z'}(l) \longrightarrow \omega_X(k-l+7) \longrightarrow \omega_R \longrightarrow 0;\label{five}
		\end{gather}
		note that $Z' \subset X \subset Z$. As $X$ is Cohen-Macaulay and $R$ has dimension zero, \eqref{five} implies that $Z' = X$ and it follows from \eqref{four} that 
		\[
		\coker\eta = \mathscr{I}_{X/Z} (l).
		\]
		If $\dim\coker\eta=0$ then $X=\sing_1(\sD)$ which implies that $(\eta)_0 = \sing_0(\sD)$. Since $T_\sD\simeq F^*$, the exact sequence in the statement comes from dualizing the exact sequence in display \eqref{seq:one}, using the hypothesis $\dim\coker\eta=0$. 
	\end{proof}
	
	\begin{remark}\label{rem:eta=0}
		We make two observations on the previous argument.
		\begin{enumerate}
			\item $\eta=0$ if and only if $F^*=\op3(1-k)\oplus\op3(k+3-l)$, so that $l-k \geq 2$;
			\item If $\eta\ne0$ and $Y:=\sing_1(\sG)$ is irreducible and reduced, then the hypothesis $\dim\coker\eta=0$ is automatically satisfied.
		\end{enumerate}
	\end{remark}
	
	Also note that the codimension one distribution $\sD$ constructed above, and for which $\dim\coker\eta=0$ holds, has the following invariants:
	\begin{align}
		\begin{split}\label{eq:chern dist from folbc}
			c_1(T_\sD) & = 4-l ;\\
			c_2(T_\sD) & = l(k-1) + 6 - c_2(N_\sG^\ast); \\ 
			c_3(T_\sD) & = c_3(N_\sG^\ast) + (1-k - l)c_2(N_\sG^\ast) + (k^2 + 2k + 3)l - 4.
		\end{split}
	\end{align}
	They can be calculated from \eqref{eq:chern-deg,pa} using that $\sing_1(\sD) = (\sigma)_0$. Since we are interested in degree two distributions, we specialize the above formulas to $l=4$:
	\begin{align}\label{eq:chern dist from folbc deg2}
		\begin{split}
			c_2(T_\sD) & = 4k+2 - c_2(N_\sG^\ast); \\
			c_3(T_\sD) & = c_3(N_\sG^\ast) - (k +3)c_2(N_\sG^\ast) + 4k^2 + 8k + 8. 
		\end{split}
	\end{align}
	
	
	\subsection{Distributions from syzygies} \label{subsect:syzygies}
	
	Another way to construct codimension one distributions, maybe the most traditional one, is to give an explicit twisted $1$-form that defines it. We will see how to construct $1$-forms with specified vanishing locus, so that the distribution has the desired invariants. We will proceed by studying some homogeneous ideals as in \cite[Section 4]{CJMdeg1}.
	
	Recall that a degree $d$ codimension one distribution $\sD$ may be given by a homogeneous $1$-form $\omega_\sD = A_0 dx + A_1dy + A_2 dz + A_3dw$. The singular scheme $Z$ of $\sD$ is defined by the saturated ideal
	\[
	I_Z = I^{sat} = \bigoplus_{l\in \mathbb{Z}} H^0\! \left(\mathscr{I}_Z(l)\right),
	\]
	where $I = (A_0, A_1,A_2,A_3) \subset \mathbb{C}[x,y,z,w]$. As the $A_j$ have degree $d+1$, restrictions are imposed on the possible subschemes of $\mathbb{P}^3$ that can fit into the singular loci of distributions. 
	
	\begin{lemma}\label{lem:incsat}
		Let $C\subset \mathbb{P}^3$ be a subscheme with saturated ideal $I_C$ and let $D\supset C$ be the subscheme defined by $(I_C)_{\leq d+1}$, the ideal generated by the elements of $I_C$ of degree $\leq d+1$. If $\sD$ is a distribution as above and $C\subset Z$, then $D\subset Z$.
	\end{lemma}
	
	\begin{proof}
		Since $C\subset Z$ we have $I_Z \subset I_C$; in particular $(A_0, A_1,A_2,A_3)\subset I_C$. In fact,
		$(A_0, A_1,A_2,A_3)\subset (I_C)_{\leq d+1}$ and the result follows from the inclusion of the saturations.
	\end{proof}
	
	A direct consequence of this lemma is a bound on the genus of double and triple lines that can be included in the singular scheme of a distribution.
	
	\begin{corollary}\label{cor:2line}
		Let $\sD$ be a degree $d$ distribution on $\mathbb{P}^3$ and let $C$ be a double line of genus $g\leq 0$. If $g<-d$ then $\sD$ is singular along the second infinitesimal neighborhood of $C_{red}$, i.e., the curve defined by $\left(I_{C_{red}}\right)^2$.
	\end{corollary}
	
	\begin{proof}
		Up to a linear change of coordinates, $C_{red} = \{x=y=0\}$ and the ideal of $C$ is $I_C = (x^2, xy,y^2, xp+yq)$ where $p$ and $q$ have degree $-g$, see \cite{N}. 
		
		If $g <-d$, then $\deg (xp+yq) \geq d+2$ and Lemma \ref{lem:incsat} implies that $I_Z \subset (x,y)^2$.
	\end{proof}
	
	The second infinitesimal neighborhood of a line has degree three. This will be useful to bound $c_3(T_\sD)$ in some cases. Next we describe the restriction on triple lines. 
	
	Fix $C_{red} = \{x=y=0\}$. According to \cite{N}, a triple structure on $C_{red}$ is described by a pair of numbers $(a,b)$ and falls in one of two cases: 
	\begin{enumerate}
		\item $a=-1$ then $C$ is a curve of genus $p_a(C) = 1-b$ given by
		\[
		I_C = (x^2, xy, y^3, xq-y^2p),
		\]
		where $\deg p + 1= \deg q = b$.
		\item If $a\geq 0$ then $C$ is a curve of genus $p_a(C) = -2-3a-b$ given by
		\[
		I_C = (x,y)^3+ (  x(xg-yf),y(xg-yf),p(xg-yf) - rx^2-sxy-ty^2),
		\]
		where $\deg f = \deg g = a+1$ and $\deg p = b$. 
	\end{enumerate}

	\begin{corollary}\label{cor:3line}
		Let $\sD$ be a degree $d$ distribution on $\mathbb{P}^3$ and let $C$ be a triple line of type $(a,b)$. If either $a=-1$ and $b \geq d$; or $a\geq 0$ and $a+b \geq d$ then $\sing(\sD)$ contains a multiple structure on $C_{red}$ of degree at least $4$.
	\end{corollary}
	
	\begin{proof}
		If $a=-1$ and $b\geq d$ then $\sing(\sD)$ must contain the curve given by the ideal $(x^2, xy, y^3)$, which has degree $4$.
		
		If $a\geq 0$ and $a+b \geq d$ then $\sing(\sD)$ must contain the curve given by the ideal $(x,y)^3 +(x(xg-yf),y(xg-yf))$, which has degree at least $4$. 
	\end{proof}

	Finally we recall the following result, \cite[Proposition 4.4]{CJMdeg1}, that exhibits a correspondence between the $1$-forms a closed subscheme and the linear syzygies of its homogeneous ideal.
	
	\begin{proposition}[\cite{CJMdeg1}]\label{prop:syzygies}
		Let $Z \subset \mathbb{P}^n$ be a closed subscheme and let $d\geq 0$ be an integer. Suppose that $Z$ is not contained in a hypersurface of degree less than or equal to $d$. Then there exists a linear isomorphism between the spaces of degree $d+2$ twisted $1$-forms singular at $Z$ and linear first syzygies of the homogeneous ideal $I_Z$. 
	\end{proposition}
	
	To produce the $1$-forms we fix a minimal generating set $\{F_0, \dots , F_r\}$ for $I_Z$ and consider a linear first syzygy $\{G_0 , \dots, G_r\}$. Then we define 
	\[
	\omega = F_0 dG_0 + \dots + F_rdG_r;
	\]
	it is clear that $\omega$ is homogeneous and descends to $\mathbb{P}^n$ since $F_0G_0 + \dots + F_rG_r =0$. 
	Note that, for general ideals, one may also use higher degree syzygies and replace $dG_j$ by $\frac{1}{\deg G_j} dG_j$ in the definition of $\omega$. But then $\omega$ may be non-homogeneous and the syzygies must be chosen wisely. We used Macaulay2 \cite{M2} to do such computations.

	\subsection{Morphisms to the tangent bundle}
	Let $F$ be a rank 2 reflexive sheaf on $\p3$, with $c_1(F)=0$. Assuming that $\Hom(F,\tp3)\ne0$, we would like to find conditions that guarantee the existence of a monomorphism $F\into\tp3$ with torsion free cokernel. One criterion is given by \cite[Corollary A.4]{CCJ1}, which we now recall.
	
	\begin{lemma}[\cite{CCJ1}]\label{lem:bertini}
		Let $F$ be a globally generated rank $2$ reflexive sheaf on $\p3$. Then $F^\vee(1)$ is the tangent sheaf of a codimension one distribution $\sD$ of degree $c_1(F)$ with $c_2(T_\sD) = c_2(F)- c_1(F) + 1$, and $c_3(T_\sD) = c_3(F)$.
	\end{lemma}
	
	In particular, if $F(1)$ is globally generated, then there exists $\phi\in\Hom(F,\tp3)$ such that $\coker\phi$ is torsion free; if, in addition, $c_1(F)=0$, then the corresponding codimension one distribution $\sD$ has degree equal to $c_1(F(1))=2$, and satisfies $(c_2(T_\sD),c_3(T_\sD))=(c_2(F),c_3(F))$.
	
	More generally, assume that $\phi\colon F \to \tp3$ is not injective; it then follows that $\im\phi$ must be a torsion free sheaf, which implies that $\ker \phi$ is reflexive of rank 1, therefore $\ker\phi\simeq\op3(-k)$ and $\im\phi\simeq \mathscr{I}_C(k)$ for some curve $C$. The stability of $F$ forces $k\ge1$ and the natural inclusion $\im\phi\into\tp3$ induces a nontrivial section in $H^0(\tp3(-k))$, thus $k = 1$. Then $\phi$ decomposes as
	\[
	F \stackrel{s^\vee}{\longrightarrow} \op3(1) \stackrel{\nu}{\longrightarrow} \tp3
	\]
	where $s\in H^0(F(1))$ and $\nu \in H^0(\tp3(-1))$. This factorization allows us to prove the following lemma.
	
	\begin{lemma}\label{lem:injective maps}
		Let $F$ be a $\mu$-stable rank 2 reflexive sheaf on $\p3$ such that $c_1(F)=0$ and $\Hom(F,\tp3) \neq \{0\}$. Then there exists an injective morphism $\phi\colon F \hookrightarrow \tp3$ if and only if either
		\begin{enumerate}
			\item $h^0(F(1)) \neq 1$; 
			\item $h^0(F(1)) = 1$ and $\hom(F,\tp3) > 4$.
		\end{enumerate}
		Moreover, if $h^0(F(1)) = 0$ then every morphism in $\Hom(F,\tp3)$ is injective;
	\end{lemma}
	
	\begin{proof}
		Consider the map $\xi\colon H^0(F(1)) \otimes H^0(\tp3(-1)) \rightarrow \Hom(F,\tp3)$ defined by $\xi(s \otimes \nu) = s^\vee \nu$; clearly $\xi$ is injective. Let $\Sigma \subset H^0(F(1)) \otimes H^0(\tp3(-1)) $ be the locus of the decomposable elements, i.e., those of the form $s\otimes \nu$. From our previous discussion, the locus of maps in $\Hom(F,\tp3)$ that are not injective is precisely $\xi(\Sigma)$. 
		
		If $h^0(F(1)) = 0$ then $\Sigma = \emptyset$ and we are done. If $h^0(F(1)) = 1$ we have the other extremal case $\Sigma = H^0(F(1)) \otimes H^0(\tp3(-1))$, then there exists an injective morphism if and only if $\hom(F,\tp3) > \dim \Sigma = 4$. Finally, if $h^0(F(1)) \geq 2$ then $\Sigma$ is the affine cone of a Segre variety and
		\[
		\hom(F,\tp3) \geq 4h^0(F(1)) > h^0(F(1)) + 3 = \dim \Sigma.
		\]
		Thus there exist injective morphisms.
	\end{proof}
	
	Now we analyse the morphisms that are injective but have torsion in their cokernels. Let $\phi\colon F\into \tp3$ be a monomorphism whose cokernel is not torsion free, and let
	\begin{equation}\label{sheaf p}
		P:= \ker\{ \coker\phi\longrightarrow(\coker\phi)^{\vee\vee} \} \end{equation}
	be the maximal torsion subsheaf of $\coker\phi$. The quotient $(\coker\phi)/P$ is a torsion free sheaf of rank 1, so it must be of the form $\mathscr{I}_Z(d'+2)$ for 1-dimensional scheme $Z$ and some $d'\ge0$. We end up with the following commutative diagram
	\begin{equation}\label{saturation}
		\begin{aligned}
			\xymatrix{
				&       &      & 0 \ar[d]  & \\
				& 0\ar[d]     &      & P \ar[d]  & \\
				0\ar[r] & F \ar[d] \ar[r]^-{\phi} & \tp3 \ar@{=}[d] \ar[r] & K \ar[r] \ar[d] & 0 \\
				0\ar[r] & F'\ar[r]^-{\phi'} \ar[d]  & \tp3 \ar[r]   & \mathscr{I}_Z(d'+2) \ar[r] \ar[d] & 0 \\
				& P \ar[d]      &      & 0    & \\
				&0
			}  
		\end{aligned}
	\end{equation}
	where $K:=\coker\phi$ and $F':=\ker\{\tp3\onto K\onto K/P\}$. 
	
	The second row of the previous diagram defines a codimension one distribution $\sD$ of degree $d'$, called the \emph{saturation} of the monomorphism $\phi\colon F\to \tp3$; note that $T_\sD=F'$.
	
	\begin{lemma}\label{lem:hom saturation}
		Let $F$ be a rank 2 reflexive sheaf on $\p3$, and let $\phi \colon F\to \tp3$ be a monomorphism whose cokernel is not torsion free. If $\sD$ is the saturation of $\phi$, then:
		\begin{enumerate}
			\item The sheaf $P$ defined in display \eqref{sheaf p} has pure dimension 2.
			\item $\deg(\sD)<2-c_1(F)$;
			\item $\hom(T_\sD,\tp3)\le\hom(F,\tp3)$.
		\end{enumerate}
	\end{lemma}
	\begin{proof}
		For the first item it suffices to prove that $\inext^3(P,\op3) = 0$ and $\inext^2(P,\op3)$ has dimension $0$. Consider the diagram in display \eqref{saturation}. Dualizing the bottom row and using that $F' = T_\sD$ is reflexive we get $\inext^3(\mathscr{I}_Z(d'+2),\op3) = 0$. Then, dualizing the rightmost column we get 
		\[
		\begin{split}
			\inext^2(K,\op3) \longrightarrow \inext^2(P,\op3) \longrightarrow 0 ; \\ 0 \longrightarrow \inext^3(K,\op3) \longrightarrow \inext^3(P,\op3)\longrightarrow 0. 
		\end{split}
		\]
		On the other hand, we dualize the top row to get $\dim \inext^2(K,\op3) =0$ and $\inext^3(K,\op3) = 0$; using this information in the sequences above we show that $P$ is pure of dimension two. Therefore $c_1(P)>0$ and
		\[
		\deg(\sD) = 2 - c_1(T_\sD) = 2 - (c_1(F)+c_1(P)) < 2 - c_1(F).
		\]
		
		Finally, we apply the functor $\Hom(\ \cdot \ ,\tp3)$ to the leftmost column and get
		\[
		0\longrightarrow \Hom(T_\sD,\tp3) \longrightarrow \Hom(F,\tp3),
		\]
		which proves the third item.
	\end{proof}
	
	In general, it is hard to describe the possible saturations for a given morphism and one may rely on ad hoc methods. In our case, it will be useful to know the the dimensions of $\Hom(T_\sD, \tp3)$ for a distribution of degree $0$ or $1$. These dimensions were obtained in \cite{CCJ1} and we summarize them in Table \ref{tab:hom table deg <2}.
	
	\begin{table}[ht] \centering
		\begin{tabular}{|c|c|c|}
			\hline
			$\deg(\sD)$  & $(c_2(T_\sD),c_3(T_\sD))$ & $\hom(T_\sD,\tp3)$ \\ \hline \hline
			\multirow{2}{*}{0} & (2,0)     & 1   \\ \cline{2-3} 
			& (1,0)     & 8   \\ \hline 
			\multirow{4}{*}{1} & (3,5)     & 1   \\ \cline{2-3} 
			& (2,2)     & 5   \\ \cline{2-3} 
			& (1,1)     & 12   \\ \cline{2-3} 
			& (0,0)     & 19   \\ \hline
		\end{tabular}
		\caption{Possible values of $\hom(T_\sD,\tp3)$ for all possible $\sD$ of degree $\leq 1$.}
		\label{tab:hom table deg <2}
	\end{table}


	\section{Distributions with \texorpdfstring{$c_2(T_\sD) \leq 1$}{c2 <=1}}\label{sect:c2<=1}
	
	We start the description with the cases $c_2(T_\sD) = -1, 0$ or $1$. For $c_2(T_\sD) = -1$ we have that $T_\sD = \op3(1) \oplus \op3(-1)$. In particular, the map $\phi \colon T_\sD \hookrightarrow \tp3$ yields vector fields $\nu_0:= \phi(1,0)$ and $\nu_2:= \phi(0,1)$, of respective degrees $0$ and $2$, that generate $\sD$. Choosing homogeneous coordinates $(x:y:z:w)$, we may assume that $\nu_0 = \frac{\partial}{\partial w}$ and $\nu_2 = A\frac{\partial}{\partial x} + B \frac{\partial}{\partial y } + C\frac{\partial}{\partial z}$, for some degree $2$ polynomials $A,B$ and $C$. Then $[\nu_0, \nu_2]= 0 $, i.e. $\sD$ is integrable, if and only if $A,B$ and $C$ do not depend on the variable $w$. These are the so called linear pullback foliations $S(3)$. Of course, one can easily produce a non-integrable distribution setting $\nu_0 = \frac{\partial}{\partial w}$ and choosing $\nu_2$ depending on $w$. 
	
	As $T_\sD$ splits we can use the surjection $\op3(1)^{\oplus4} \twoheadrightarrow \tp3$ to construct a free resolution for $C= \sing(\sD)$.
	\begin{equation}\label{seq:ressing-1}
		0\longrightarrow \op3(-5) \oplus \op3(-4) \oplus \op3(-3) \longrightarrow \op3(-3)^{\oplus4} \longrightarrow \mathscr{I}_C \longrightarrow 0.
	\end{equation}
	In particular, $C$ is arithmetically Cohen-Macaulay (ACM). This is true whenever $T_\sD$ splits and in a more general setting, see \cite[Theorem 1]{CJV}.

	If $c_2(T_\sD) \geq 0$, we have $\mu$-semistability, recall Theorem \ref{thm:stability}. In particular, for $c_2(T_\sD) = 0$, this implies that $c_3(T_\sD) = 0$, hence $T_\sD = \op3 \oplus \op3$. In this case, we get from $\phi \colon T_\sD \hookrightarrow \tp3$ two linear vector fields $\nu$ and $\nu'$. Clearly, the choice of any two (linearly independent) vector fields defines a distribution. The integrable cases, i.e. $[\nu, \nu'] = a \nu + b\nu'$ with $a,b\in \mathbb{C}$, come from representations of two-dimensional complex Lie algebras. Among these we find logarithmic foliations $L(1,1,1,1)$, that come from representations of the trivial algebra, and the exceptional foliations $E(3)$ that come from representations of the affine Lie algebra $\mathfrak{aff}(\mathbb{C})$. 
	
	As above $T_\sD$ splits and we can build a free resolution 
	\begin{equation}\label{seq:ressing0}
		0\longrightarrow \op3(-4)^{\oplus 3} \longrightarrow \op3(-3)^{\oplus4} \longrightarrow \mathscr{I}_C \longrightarrow 0
	\end{equation}
	for $C= \sing(\sD)$.
	
	We summarize the previous discussion in the following proposition.
	
	\begin{proposition}
		Let $\sD$ be a codimension one distribution on $\p3$ of degree $2$. If $c_2(T_\sD)\leq 0$ then one of the following holds.
		\begin{enumerate}
			\item $T_\sD = \op3(1) \oplus \op3(-1)$ and $\sing(\sD)$ is an ACM curve as in \eqref{seq:ressing-1};
			\item $T_\sD = \op3 \oplus \op3$ and $\sing(\sD)$ is an ACM curve as in \eqref{seq:ressing0}.
		\end{enumerate}
	\end{proposition}

	For $c_2(T_\sD) = 1$ we have two possibilities: either $c_3(T_\sD) = 0$ and $T_\sD$ is $\mu$-stable; or $c_3(T_\sD) = 2$ and $T_\sD$ must be strictly $\mu$-semistable. Indeed, this holds for any reflexive sheaf with such Chern classes, see \cite[Lemma 2.1]{Ch}. Moreover, if $c_3(T_\sD) = 0$ then $T_\sD$ is a null correlation bundle, see \cite[Lemma 4.3.2]{OSS}. We will show, in the next two results, that distributions with these invariants do exist. We recall that a curve is arithmetically Buchsbaum if its Rao module $H_\ast^1(\mathscr{I}_C)$ is annihilated by the irrelevant ideal, see \cite{Am}; in particular, this is true if $H_\ast^1(\mathscr{I}_C)$ is supported in only one degree.
	
	\begin{proposition}\label{prop:case10}
		Let $N$ be a null correlation bundle on $\mathbb{P}^3$. Then there exists a distribution $\sD$ such that $T_\sD = N$. Moreover, its singular scheme is an arithmetically Buchsbaum curve $C$ of degree $5$ and arithmetic genus $1$.
	\end{proposition}
	
	\begin{proof}
		As $N(1)$ is a quotient of $\Omega_{\p3}^1(2)$, it is globally generated. We then apply Lemma \ref{lem:bertini} to show that there exists a distribution $\sD$ with tangent sheaf $T_\sD =N$. Note that from \eqref{eq:chern sing1 deg2} $C= \sing(\sD)$ is a curve of degree $5$ and genus $1$. From the sequence
		\[
		0 \longrightarrow N \longrightarrow\tp3 \longrightarrow \mathscr{I}_C(4) \longrightarrow 0
		\]
		we get, for $l\in \mathbb{Z}$,
		\[
		0 \longrightarrow H^1(\mathscr{I}_C(4+l)) \longrightarrow H^2(N(l)) \longrightarrow H^2(\tp3(l));
		\]
		on the other hand, $H^2(N(l)) = H^2(\tp3(l-1))$. Therefore, $h^1(\mathscr{I}_C(l)) = 0$ for $l\neq 1$ and $h^1(\mathscr{I}_C(1)) = 1$ and $C$ is arithmetically Buchsbaum. 
	\end{proof}
	
	We note that the spectrum of a null correlation bundle is $\{0\}$; it follows from the fact that $h^1(N(l)) =0$ for $l\neq -1$ and $h^1(N(-1)) = 1$.

	\begin{proposition}\label{prop:case12}
		Let $F$ be $\mu$-semistable rank 2 reflexive sheaf on $\p3$ with Chern classes $c_1(F) = 0$, $c_2(F) = 1$ and $c_3(F) = 2$. 
		Then $F$ can be realized as the tangent sheaf of a codimension one distribution of degree 2. Moreover, $\sing_1(\sD)$ is an ACM curve of degree 5 and genus 2 contained in a quadric. 
	\end{proposition}
	
	\begin{proof}
		Let $F$ be as in the statement. Consulting \cite[Table 2.3.1]{Ch} we check that: $h^1(F(p)) = 0$ for every $p\in \mathbb{Z}$; $h^3(F(p)) = 0$ for $p\geq -3$; and 
		\[
		h^2(F(p)) = \begin{cases}
			0, & p \geq -1\\ 1, & p=-2\\ 2, & p\leq -3
		\end{cases}.
		\]
		Then the Castelnuovo--Mumford criterion implies that $F(1)$ is globally generated. Owing to Lemma \ref{lem:bertini}, there exists a codimension one distribution $\sD$ of degree 2 whose tangent sheaf is precisely $F$.
		
		For the second claim, let $C:=\sing_1(\sD)$; it follows from the equations in display \eqref{eq:chern sing1 deg2} that $C$ has degree $5$ and arithmetic genus $2$. We must show that $h^1(\mathscr{I}_C(p))=0$ for every $p$, so that $C$ is ACM.
		
		From the definition of $C$ in the sequence in display \eqref{seq:sing decomp dist}, we have that 
		\begin{equation} \label{seq:hh decomp}
			H^0(\mathscr{I}_Z (l)) \rightarrow H^0(\mathscr{I}_C (l)) \rightarrow H^0(U) \rightarrow H^1(\mathscr{I}_Z (l)) \rightarrow H^1(\mathscr{I}_C (l)) \rightarrow 0 
		\end{equation}
		and $h^0(U) = c_3(F) = 2$. Also recall that $h^1(\mathscr{I}_Z (l)) = h^2(F(l-4))$ for $l\neq 0$ and $h^1(\mathscr{I}_Z) \leq h^2(F(-4))$.
		
		For $l\geq 3$, we see that $h^1(\mathscr{I}_C(l)) \leq h^1(\mathscr{I}_Z (l)) = h^2(F(l-4)) = 0$.
		
		For $l\leq 1$, we have $h^0(\mathscr{I}_C(l))=0$, since $C$ cannot be planar. On the other hand, $h^1(\mathscr{I}_Z (l)) \leq h^2(F(l-4)) = 2$. This implies that $H^0(U) \simeq H^1(\mathscr{I}_Z (l))$ hence $h^1(\mathscr{I}_C(l))=0$.
		
		For $l=2$, we have $h^1(\mathscr{I}_Z(2)) = h^2(F(-2)) = 1$; and as $h^1(F(-2)) = 0$ we also have $h^0(\mathscr{I}_Z(2))= 0$. From \eqref{seq:hh decomp} we only need to show that $h^0(\mathscr{I}_C(2)) = 1$; the only other possibility is $h^0(\mathscr{I}_C(2)) = 2$.
		
		Suppose that $h^0(\mathscr{I}_C(2)) = 2$. Since $\deg C = 5$, then $H^0(\mathscr{I}_C(2))$ is spanned by $ff_1$ and $ff_2$ for some linear polynomials $f, f_1, f_2 \in H^0(\op3(1))$. First we claim that this would imply that $C$ contains a planar subcurve of degree $4$ which is impossible, due to Lemma \ref{lem:incsat}. Therefore $h^0(\mathscr{I}_C(2)) = 1$ and $C$ is ACM which concludes the proof.
		
		Now we prove the claim. We have two cases:
		\begin{enumerate}
			\item $\{f, f_1, f_2\}$ is linearly independent then $C$ is the union of a plane quartic $C'$ and the line $L = \{f_1 = f_2 =0\}$;
			\item $\{f, f_1, f_2\}$ is linearly dependent, so we may assume $f_2= f$; then $C$ is a curve in the double plane defined by $f^2$.
		\end{enumerate}
		In the second case, owing to \cite[Proposition 2.1]{HSdp}, we have
		\[
		0 \longrightarrow \mathscr{I}_Y(-H) \longrightarrow \mathscr{I}_C \longrightarrow \mathscr{I}_{V/H}(-P) \longrightarrow 0,
		\]
		where $V\subset Y \subset P \subset H = \{f=0\}$, for curves $Y$ and $P$ and a $0$-dimensional subscheme $V$. Moreover, $\deg C = \deg Y + \deg P$ and $P$ is the largest curve contained in $C\cap H$. Also, $Y$ is the residual intersection of $C$ and $H$ hence it must be the line $\{f = f_1 = 0\}$. Then $P$ is a plane quartic contained in $C$.
		
	\end{proof}
	
	Among these distributions we find logarithmic foliations in $L(1,1,2)$, which is singular at a union of two conics and a line (plus two points); in particular, that unique quadric from the statement of Proposition \ref{prop:case12} is reducible. It is not hard to find non-integrable distributions singular at a smooth curve contained in a smooth quadric.
	
	Finally, the only possible spectrum of a $\mu$-semistable sheaf $F$ with $c_1(F) = 0$, $c_2(F) = 1$ and $c_3(F) = 2$ is $\{-1\}$.


	\section{Distributions with \texorpdfstring{$c_2(T_\sD)=2 $}{c2 =2}}\label{sect:c2=2}
	
	If $\sD$ is a degree two distribution with $c_2(T_\sD) = 2$ then the bound in \eqref{eq:bound c3} implies $c_3(T_\sD) = 0, 2$ or $4$. From Theorem \ref{thm:stability} we know that $T_\sD$ is $\mu$-stable except for the case $c_3(T_\sD)= 4$ where it can be strictly $\mu$-semistable.
	
	The existence of degree $2$ codimension one distributions $\sD$ such that $c_2(T_\sD)=2$ and $c_3(T_\sD) = 0$ or $4$ was established in \cite[Theorem 9.5]{CCJ1} and \cite[Section 11.5]{CCJ1}, respectively. We describe them here for the sake of completeness.
	
	If $c_3(T_\sD) = 0$ then $T_\sD$ is an instanton bundle of charge two. Indeed, any $\mu$-stable bundle with these Chern classes satisfies $h^1(T_\sD(-2)) = 0$, see \cite[Lemma 9.4]{H3}. The singular scheme $\sing(\sD)$ is a curve $C$ of degree $4$ and genus $-1$. The Hilbert scheme $\calh(4,-1)$ has three irreducible components, described in \cite[Proposition 6.1]{NS:deg4}. Since $T_\sD$ is an instanton, we know that $h^0(\mathscr{I}_C(2)) = h^1(T_\sD(-2)) = 0$; hence $C$ can only be the disjoint union of a line and a twisted cubic (or some degeneration of that). We note that as $h^1(T_\sD(-1)) = 2$ the only possible spectrum is $\{0,0\}$.
	
	\begin{example}
		Consider the curve $C$ defined by 
		\[
		I_C = (z^2-yw, yz-xw, y^2-xz) \cap (x,w).
		\]
		Using its syzygies we could find the $1$-form
		\begin{align*}
			\omega= & (xz^2-xyw+y^2w-xzw+yzw+z^2w-xw^2-yw^2)dx\\ 
			&+(-xyz+xz^2+x^2w-xyw+yzw+z^2w-xw^2-yw^2)dy\\
			&+(xy^2-x^2z-xyz+x^2w-y^2w+xzw-yzw+xw^2)dz \\
			&+(-xyz-2xz^2+x2w+2xyw+y^2w-xzw)dw
		\end{align*}
		which is singular precisely at $C$. Thus $T_\sD$ is an instanton bundle of charge two. 
	\end{example}

	If $c_3(T_\sD) = 4$ then $\sing(\sD)$ is composed by $4$ points and a curve $C$ of degree $4$ and genus $1$. Any such curve is (some degeneration of) an elliptic quartic curve, see \cite{NS:deg4}. In particular, $h^0(\mathscr{I}_C(2)) = 2$ and there exists a unique pencil of quadrics containing $C$. If $C$ is general, this pencil defines a degree two rational foliation in $R(2,2)$; then $\sing_0(\sD)$ is composed by the vertices of the cones appearing in this pencil. Moreover, $T_\sD$ is stable in this case. Conversely, any stable sheaf with these Chern classes can be realized as the tangent sheaf of a distribution that, in general, is not integrable.
	
	\begin{proposition}
		Let $F$ be a $\mu$-stable rank two reflexive sheaf on $\mathbb{P}^3$ with \linebreak $(c_1(F),c_2(F),c_3(F))=(0,2,4)$. Then there exists a distribution $\sD$ such that $T_\sD = F$. Moreover, $\sing_1(\sD)$ is a, possibly degenerated, elliptic quartic.
	\end{proposition}
	
	\begin{proof}
		Owing to \cite[Table 2.12.2]{Ch} and Castelnuovo--Mumford criterion, $F(1)$ is globally generated. Then, due to Lemma \ref{lem:bertini}, there exists a codimension one distribution of degree 2 whose tangent sheaf is precisely $F$. The assertion about the singular scheme follows from our previous discussion.
	\end{proof}
	
	We note that we can also find distributions with $T_\sD$ strictly $\mu$-semistable. 
	
	\begin{example}\label{ex:c2=2 c3=4 semistable}
		Consider the curve $C$ given by the ideal 
		\[
		I_C = (w^2 - (y-x)z, z^2-xy)
		\]
		and the double point $I_{p_1} = (x,y^2,w)$. Using the syzygies we find the $1$-form
		\begin{align*}
			\omega = & (-xy2+xz2-xyw+xzw-yzw+z2w+zw2+w3)dx\\ &+(x2y-xz2+xzw-yzw+w3)dy+(-x2z+xyz-xw2)dz \\&+(x2y-x2z+y2z-xz2-xw2-yw2)dw
		\end{align*}
		singular at the prescribed singular scheme and also at the double point given by $I_{p_2} =(y + w, 2x - z - w, (z+w)^2)$. Computing the kernel of $\omega$ we see that $T_\sD$ has a free resolution 
		\[
		0 \longrightarrow \op3(-4) \longrightarrow \op3(-3)^{\oplus4} \longrightarrow \op3 \oplus \op3(-2)^{\oplus4} \longrightarrow T_\sD \longrightarrow 0.
		\]
		In particular, $h^0(T_\sD) = 1$ and $T_\sD$ is strictly $\mu$-semistable. It is also not hard to find a degree one foliation by curves tangent to this distribution.
	\end{example}
	
	We note that, from Remark \ref{rem:properties of spectrum}, the only possible spectrum for a $\mu$-semistable sheaf $F$ such that $c_1(F) = 0$, $c_2(F) = 2$ and $c_3(f)=4$ is $\{-1,-1\}$.
	
	Now consider the case $c_3(T_\sD) = 2$. Then the singular scheme is composed by $2$ points and a curve of degree $4$ and genus $0$. The Hilbert scheme $\calh(4,0)$ has two irreducible components, whose generic points are given by either a rational quartic curve, or a disjoint union of a plane cubic and a line. We now argue that the second case does not occur. 
	
	The homogeneous ideal of a disjoint union of a plane cubic and a line is a product $I_C = (h,f)(h_1,h_2)$ where $h,h_1$ and $h_2$ have degree $1$ and $f$ has degree $3$. Then $(I_C)_{\leq 3} = (hh_1, hh_2)$ and, due to Lemma \ref{lem:incsat}, $\sing(\sD)$ would have codimension one.
	
	For $C$ a rational quartic curve there exist distributions; we will now provide an example.
	
	\begin{example}\label{ex: c2 = c3 = 2}
		Fix the rational quartic curve given by 
		\[
		I_C = (yz - xw, z^3 - yw^2 , xz^2 - y^2w, y^3 - x^2z),
		\]
		and the point $(1:0:0:1)$. Using the syzygies we find 
		\begin{align*}
			&\quad\omega = (2xz^2-z^3-2y^2w+yw^2)dx \\&+(-2y^2z+2xz^2+2yz^2+2xyw-2y^2w-2xzw)dy\\ &+(2y^3-2x^2z-2xyz-2y^2z+xz^2+2x^2w+2xyw-y^2w+2yzw-2xw^2)dz \\&+(2y^3-2x^2z+y^2z-2yz^2-xyw+2xzw)dw
		\end{align*}
		which is singular at the prescribed scheme and at the point $(1:1:0:2)$. Therefore $T_\sD = \ker \omega $ has the desired Chern classes.
	\end{example}
	
	Also, from Remark \ref{rem:properties of spectrum}, we see that the only possible spectrum for a $\mu$-semistable sheaf $F$ such that $c_1(F) = 0$, $c_2(F) = 2$ and $c_3(F)=2$ is $\{-1,0\}$.


	\section{Distributions with \texorpdfstring{$c_2(T_\sD)=3 $}{c2 =3}}\label{sect:c2=3}
	
	In this case, we know that $T_\sD$ is $\mu$-stable, so that $c_3(T_\sD) \in \{0,2,4,6,8\}$. We start by recalling the case $c_3(T_\sD)=8$, which is discussed in \cite[Section 11.6]{CCJ1}. Next we will describe the cases $c_3(T_\sD) <8$.
	
	A general stable rank 2 reflexive sheaf $F$ such that $c_1(F) = 0$, $c_2(F) = 3$ and $c_3(F) = 8$ arises as the tangent sheaf of degree two distribution. This is due to $F(1)$ being globally generated, see \cite[Lemma 3.8]{Ch}, and Lemma \ref{lem:bertini}. The singular scheme is composed by $8$ points and a curve of degree 3 and genus 1, which can only be a plane cubic, see \cite[Proposition 3.1]{N}. Furthermore, $\{-2,-1,-1\}$ is the only possible spectrum for such sheaves. 
	
	Among these distributions we find rational foliations in $R(1,3)$. These are given by pencils of cubics generated by a triple plane and a general cubic surface, so that they intersect in codimension two. For a given plane cubic curve $C$ there exists a unique pencil of cubic surfaces whose base locus is $C$; thus we have a unique rational foliation for a general $C$. The $0$-dimensional part of the singular scheme is the vertex of the unique cone in this pencil. 
	
	\subsection{ Case \texorpdfstring{$ 2\leq c_3(T_\sD) \leq 6$}{0 < c3 <8} } \label{sect:c2=3 2-6}
	Consider $\calr(0,3,2l)$ the moduli space of stable rank two reflexive sheaves $F$ on $\p3$ satisfying $c_1(F) = 0$, $c_2(F) = 3$ and $c_3(F)= 2l$. We begin by noting that when $l>0$, Chang proved in \cite{Ch} that the moduli space $\calr(0,3,2l)$ is irreducible and contains a nonempty open subset
	\[
	\calr_0(0,3,2l) =\left\{ [F] \mid h^1(F(1))=0 \right\},
	\]
	for which we can prove the following result.

	\begin{proposition}\label{prop:inj gen c2=3}
		Let $l \in \{1,2,3\}$. Then every stable rank 2 reflexive sheaf $F$ with Chern classes $(c_1(F),c_2(F),c_3(F))=(0,3,2l)$ satisfying $h^1(F(1))=0$ admits a monomorphism $F\into\tp3$.
	\end{proposition}
	
	\begin{proof}
		Applying $\Hom(F, \ \cdot \ )$ to the Euler sequence and using that $h^1(F(1))=0$ we get
		\[
		0 \longrightarrow H^0(F(1))^{\oplus4} \longrightarrow \Hom(F,\tp3) \longrightarrow H^1(F) \longrightarrow 0, 
		\]
		so that $\hom(F,\tp3) = 4h^0(F(1)) + h^1(F)$. Consulting the cohomology tables \cite[Tables 3.7.1, 3.4.1 and 3.5.1]{Ch}, we get that $h^0(F(1)) = l-1$ and $h^1(F) = 3-l$, thus $\hom(F,\tp3) = 3l-1>0$. Therefore, the result follows from a direct application of Lemma \ref{lem:injective maps}.
	\end{proof}
	
	It is not clear if every such sheaf can be realized as the tangent sheaf of a distribution; i.e. it is not clear if for any given $[F]\in\calr_0(0,3,2l)$ there exist monomorphisms $F\into\tp3$ with torsion free cokernel. For the moment, we can show that this is true generically. Indeed, we only need to provide an example to show that an every sheaf in an open subset of $\calr_0(0,3,2l)$ can be realized as tangent sheaves of distributions. 
	
	\begin{example}\label{ex:c2=3 c3=6}
		If $c_3(T_\sD)= 6$, any such distribution must be singular at a curve $C$ of degree $3$ and genus $0$, plus six points. According to \cite[Lemma 1]{RP:MS}, any such curve is a, possibly degenerated, twisted cubic. Then let $C$ be given by 
		\[
		I_C = (-y^2+xz, -yz+xw, -z^2+yw)
		\]
		and add the points $(0:1:0:0)$, $(0:0:1:0)$ and $(1:1:-1:1)$. We can get the $1$-form
		\begin{align*}
			&\quad \omega = (-y^2z+yz^2+xyw-y^2w+yzw-z^2w-xw^2+yw^2)dx\\ &+(xyz+yz^2-x^2w-xzw+z^2w-yw^2)dy \\ &+(-xyz-y^2z+x^2w+xyw+2y^2w-2xzw+4yzw-4xw^2)dz \\ &+ (xy^2-x^2z-xyz-2y^2z+3xz^2-5yz^2+x^2w-xyw+y^2w+4xzw)dw
		\end{align*}
		which is singular at the prescribed scheme and also at the points $(3:0:1:0)$, $(1:-7:5:-7)$ and $(9: -31: 24: -33)$. Therefore $T_\sD = \ker \omega$ has the desired Chern classes. 
		
	\end{example}
	
	\begin{example}\label{ex:c2=3 c3=4}
		For $c_3(T_\sD)= 4$ we must have $C= \sing_1(\sD)$ a degree $3$ curve of genus $-1$. Then $C$ must be extremal, owing to \cite[Theorem 4.1]{DesP}; this means that $C$ must contain a conic, possibly degenerated. So let us fix $C$ a disjoint union of a conic and a line, given by
		\[
		I_C = (x,y) \cap (w, z^2-xy);
		\]
		then we choose the points $(1:1:1:1)$, $(4:1:2:1)$ and $(1:4:2:1)$. We get a $1$-form
		\begin{align*}
			&\quad\omega = (-xy^2+yz^2-45xyw+45y^2w+xzw-82yzw+70xw^2+11yw^2)dx\\&+(x^2y-xz^2+45x^2w-45xyw-134xzw+53yzw+115xw^2-34yw^2)dy\\&+(-x^2w+216xyw-53y^2w-211xw^2+49yw^2)dz\\&+(-70x^2w-126xyw+34y^2w+211xzw-49yzw)dw
		\end{align*}
		that also vanishes at the point $(289: 1225: -595: -3299)$. Therefore, the sheaf $T_\sD = \ker \omega$ has the desired Chern classes.
	\end{example}

	\begin{example}\label{ex:c2=3 c3=2} 
		If $c_3(T_\sD) = 2$ then $C= \sing_1(\sD)$ is a degree $3$ curve of genus $-2$. Due to \cite[Proposition 3.4]{N}, we know that one of the following holds:
		\begin{enumerate}
			\item $C$ is composed by three skew lines, possibly degenerated;
			\item $C$ is the union of a double line of genus $-3$ and a reduced line.
		\end{enumerate}
		Now recall that, by Corollary \ref{cor:2line}, 
		$\sing(\sD)$ cannot contain a double line of genus $-3$; Thus $C$ falls in the first case. So fix $C$ given by 
		\[
		I_C = (x,y) \cap (z,w) \cap (x-z,y-w)
		\]
		and add the point $P= (1:-1:1:1)$. We can compute the $1$-form
		\begin{align*}
			& \quad \omega = (xyz+y^2z+yz^2+xyw-y^2w-2xzw-4yzw+2xw^2+yw^2)dx \\ &+(-x^2z-xyz+xz^2+yz^2-x^2w+xyw+xzw-yzw)dy\\ &+ (-2xyz-y^2z+x^2w+xyw-y^2w+xzw+yw^2)dz\\ &+ (x^2z+2xyz+2y^2z-xz^2-2x^2w-xyw-yzw)dw
		\end{align*}
		that is singular at $C\cup \{P,Q\}$ where $Q =(1:-2:15:5)$. Therefore $T_\sD = \ker \omega$ has the desired Chern classes.
	\end{example}

	Now we will show that the special sheaves in $\calr(0,3,2l) \setminus \calr_0(0,3,2l)$ cannot be realized as tangent sheaves of distributions.
	
	\begin{proposition}\label{prop:special c2=3}
		Fix an integer $l \in \{1,2,3\}$ and let $\sD$ be a distribution such that $[T_\sD] \in \calr(0,3,2l)$ then $h^1(T_\sD(1))=0$.
	\end{proposition}
	
	\begin{proof}
		Recall that Lemma \ref{lem:h2=0} implies $h^1(T_\sD(1)) = h^1(\omega_C(1))= h^1(\mathscr{I}_{C}(-1))$, where $C = {\sing_1(\sD)}$, for any distribution of degree $2$. Then we only need to show that one of these other cohomologies vanish. 
		
		If $l=3$ then $C$ is a, possibly degenerated, twisted cubic. In particular, $C$ is ACM hence $h^1(\mathscr{I}_{C}(-1))=0$.

		If $l=2$ then $C$ is extremal. As $\deg C = 3$ and $p_a(C) = -1$, we have that $h^1(\mathscr{I}_{C}(l))=0$ for $l\leq -1$, see \cite{DesP}.
		
		If $l=1$, we know from Example \ref{ex:c2=3 c3=2} that $C$ is the disjoint union of three lines, possibly degenerated. From the proof of \cite[Proposition 3.4]{N} we have that $h^1(\mathscr{I}_{C}(-1))=0$.
	\end{proof}

	\begin{corollary}\label{cor:special c2=3}
		Let $l \in \{1,2,3\}$. If $[F] \in \calr(0,3,2l) \setminus \calr_0(0,3,2l)$ then there does not exist distribution with $F$ as the tangent sheaf.
	\end{corollary}
	
	Despite that we do not have distributions, many of these special sheaves admit monomorphisms to $\tp3$. For instance, using \cite[Table 3.5.1]{Ch} we see that for any $[F]\in\calr(0,3,6) \setminus \calr_0(0,3,6)$ we get $h^0(F(1)) = 3$ hence, due to Lemma \ref{lem:injective maps}, there exists a monomorphism $F\hookrightarrow \tp3$. Therefore the cokernel of any such monomorphism has torsion. 
	
	Finally, we note that since  $[T_\sD] \in \calr_0(0,3,2l)$, the spectra are completely determined. It is $\{-1,0,0\}$ for $l=1$;  $\{-1,-1,0\}$ for $l=2$; and $\{-1,-1,-1\}$ for $l= 3$, see \cite[Tables 3.7.1, 3.4.1 and 3.5.1]{Ch}.
	
	
	\subsection{Case \texorpdfstring{$ c_3(T_\sD) =0$}{c3 = 0}}\label{sec:3-0}
	
	Let us finally consider codimension one distributions such that $T_\sD$ is a locally free sheaf with $c_2(T_\sD)=3$; hence $C = \sing(\sD)$ is a curve of degree 3 and genus $-3$. From \cite[Proposition 3.5]{N} we know that either
	\begin{enumerate}
		\item $C$ is the union of a double line of genus $-4$ and a reduced line;
		\item $C$ is a disjoint union of a double line of genus $-2$ with another line.
	\end{enumerate}
	Corollary \ref{cor:2line} implies that the first case cannot occur, so $C$ belongs to the second one. Owing to \cite[Proposition 3.3]{N}, in fact to its proof, any such curve $C$ admits a free resolution
	\[
	0 \longrightarrow \op3(-5)^{\oplus3} \longrightarrow \op3(-4)^{\oplus9} \longrightarrow \op3(-3)^{\oplus7} \longrightarrow \mathscr{I}_C \longrightarrow 0.
	\]
	and satisfies $h^2(\mathscr{I}_C(l)) = 2h^1(\mathbb{P}^1, \op1(l)) + h^1(\mathbb{P}^1, \op1(l+1))$. With this information one can compute $h^i(\mathscr{I}_C(l))$ for every $i$ and $l$. Therefore, using Lemma \ref{lem:h2=0} and the exact sequence
	\[
	0 \longrightarrow T_\sD \longrightarrow \tp3 \longrightarrow \mathscr{I}_C(4) \longrightarrow0
	\]
	we can prescribe the cohomology for $T_\sD$ as in Table \ref{tab:cohomology c2=3 c3=0}. Note that the complement of this table can be computed via Serre duality. In particular, $T_\sD$ is described by the next result.
	
	\begin{table}[ht]
		\centering
		\def\arraystretch{1.5}
		\begin{tabular}{|c|c|c|c|c|c|} \hline
			& $-2$ & $-1$ & $0$ & $1$ & $l\geq 2$ \\ \hline \hline
			$h^0(T_\sD(l))$ & $0$ & $0$ & $0$ & $0$ & $\frac{1}{3}l^3 + 2l^2 + \frac{2}{3}l -4$ \\ \hline
			$h^1(T_\sD(l))$ & $0$ & $3$ & $4$ & $1$ & $0$ \\ \hline
			$h^2(T_\sD(l))$ & $0$ & $0$ & $0$ & $0$ & $0$ \\ \hline
			$h^3(T_\sD(l))$ & $0$ & $0$ & $0$ & $0$ & $0$ \\ \hline
		\end{tabular}
		\caption{Cohomology table for $T_\sD$ with $c_2(T_\sD) = 3$ and $c_3(T_\sD) = 0$. }
		\label{tab:cohomology c2=3 c3=0}
	\end{table}
	
	\begin{theorem}\label{thm:instanton3}
		Let $E$ be stable rank two vector bundle on $\mathbb{P}^3$ with $c_1(E) = 0$ and $c_2(E) = 3$. Then $E$ is the tangent sheaf of a distribution if and only if it is a generic instanton bundle $E$ of charge $3$ with a unique jumping line of order 3. 
	\end{theorem}
	
	Recall that a line $Y\subset\p3$ is said to be a jumping line of order $k \geq 1$ for $E$ if $E|_Y=\mathcal{O}_Y(-k)\oplus\mathcal{O}_Y(k)$. According to \cite[Section 1]{GS}, generic instanton bundles of charge 3 admit at most one jumping line of order 3.
	
	\begin{proof} 
		First suppose that there exists $\sD$ such that $E = T_\sD$. From Table \ref{tab:cohomology c2=3 c3=0} we have that $E$ is an instanton bundle with natural cohomology, i.e., for every $l$ at most one $h^1(E(l))$ does not vanish, see \cite{HH}. 
		Due to \cite[Lemme 1.1]{GS}, either $E$ has a unique jumping line of order three or the map $\xi\colon H^1(E) \otimes H^0(\op3(1)) \rightarrow H^1(E(1))$, given by multiplication, is a nondegenerate pairing. We will prove that the later cannot occur.
		
		Tensorizing the Euler sequence with $E$ we get
		\[
		0 \longrightarrow H^0(E\otimes \tp3) \longrightarrow H^1(E) \stackrel{\zeta}{\longrightarrow} H^1(E(1)) \otimes H^0(\op3(1))^\vee \longrightarrow H^1(E\otimes \tp3) 
		\]
		and it follows that $\zeta$ is an isomorphism if and only if $\xi$ is nondegenerate. By hypothesis, $h^0(E\otimes \tp3) = \hom(E,\tp3) \neq 0$ and $\zeta$ cannot be an isomorphism, whence $E$ has a unique jumping line of order $3$. 
		
		Now we prove the converse. Suppose that $E$ is a generic instanton bundle of charge $3$, its cohomology being given by Table \ref{tab:cohomology c2=3 c3=0}, and suppose that $E$ has a unique jumping line of order $3$. From our previous argument, there exists $\phi \colon E \rightarrow \tp3$ which is injective, due to Lemma \ref{lem:injective maps}. In fact, owing to the proof of \cite[Lemme 1.1]{GS}, the map $\zeta$ has rank two, hence $\hom(E,\tp3) = 2$. To conclude we need to prove that $\coker \phi$ is a torsion free sheaf.
		
		Suppose, by contradiction, that $\coker \phi$ has torsion and let $\sF$ be the saturated distribution; then
		\begin{equation}\label{seq:sat c2=3}
			0 \longrightarrow E \stackrel{\beta}{\longrightarrow} T_\sF \longrightarrow P \longrightarrow 0.
		\end{equation}
		According to Lemma \ref{lem:hom saturation}, $P$ is a pure sheaf of dimension two, $\deg(\sF) \leq 1$ and $\hom(T_\sF, \tp3) \leq \hom(E,\tp3) = 2$. From Table \ref{tab:hom table deg <2} we see that $\sF$ must fall in one of two cases that we will analyze now.
		
		First, $\deg(\sF) = 1$ and $c_2(T_\sF) = 3$ and $c_3(T_\sF) = 5$. In this case, $\sing(\sF)$ is a non-planar zero-dimensional subscheme of length $5$; and $P$ is supported on a (reduced) plane. Dualizing the sequence \eqref{seq:sat c2=3} we get
		\[
		\inext^1(P,\op3) \longrightarrow \inext^1(T_\sF, \op3) \longrightarrow 0.
		\]
		Thus, $\sing(\sF)$ must be contained in $\supp(P)$, which is absurd.
		
		The second case is: $\deg(\sF) = 0$ and $c_2(T_\sF) = 2$ and $c_3(T_\sF) = 0$. In this case, $T_\sF = N(1)$, where $N$ is a null correlation bundle.
		
		Consider $\mathbb{G}(1,3)\subset \mathbb{P}^5$ the Grassmannian of lines in $\mathbb{P}^3$ and let $T\subset \mathbb{G}(1,3)$ be the curve of jumping lines of order at least $2$ for $E$. Let $H \subset \mathbb{P}^5$ a hyperplane such that $H\cap \mathbb{G}(1,3)$ is the set of jumping lines for $N$. Also consider $X \subset \mathbb{G}(1,3)$ the Fano variety of lines in $\supp P$.
		
		Owing to \cite[2.2]{GS}, the moduli space of instanton bundles of charge $3$ with a jumping line of order $3$ is birational to the Hilbert scheme of rational quintic curves in $\check{\mathbb{P}}^3$. Assume that $E$ is general so that it corresponds to a smooth rational quintic $\Gamma \subset \check{\mathbb{P}}^3$ not contained in a quadric. Under this assumption $\Gamma$ has a unique $4$-secant line which is precisely $\check{L}$, for $L$ the jumping line of order $3$ of $E$. According to \cite[3.4.2]{GS}, $T$ parameterizes the trisecant lines for $\Gamma$ hence $T$ is an integral curve not contained in $X$. Also $T$ is the intersection of $7$ quadrics hence it cannot be contained in $H$. We conclude that $T \setminus (H\cup X) \neq \emptyset$, which we will use this to get a contradiction. 
		
		Let $L$ be a line in $\p3$ corresponding to $[L] \in T \setminus (H\cup X)$; as $L\not \subset \supp P$ we have that ${\rm Tor}_1(P, \mathcal{O}_L) = 0$. 
		Restricting \eqref{seq:sat c2=3} to $L$ we get
		\[
		0 \longrightarrow \mathcal{O}_L(2) \oplus \mathcal{O}_L(-2) \longrightarrow \mathcal{O}_L(1)^{\oplus2} \longrightarrow P|_L \longrightarrow 0,
		\]
		which is an absurd, since $\Hom(\mathcal{O}_L(2), \mathcal{O}_L(1)) = 0$.
	\end{proof}
	
	\begin{remark}
		We complete this section by observing that for distributions $\sD$ as above, the two lines of $C_{red} = \sing_1(\sD)_{red}$ are jumping lines for $T_\sD$: the double line of $C$ is supported on the unique jumping line of order $3$ for $T_\sD$, while the reduced line of $C$ is a jumping line of order $2$. Indeed, let $L_1$ denote the support of the double line of $C$ and $L_2$ its the reduced line; dualizing the sequences 
		\begin{gather*}
			0 \longrightarrow \mathcal{O}_{L_1}(1) \oplus \mathcal{O}_{L_2} \longrightarrow \mathcal{O}_{C} \longrightarrow \mathcal{O}_{L_1} \longrightarrow 0, \\
			0 \longrightarrow T_\sD \longrightarrow \tp3 \longrightarrow \mathscr{I}_C(4) \longrightarrow 0,
		\end{gather*}
		we obtain the composition of epimorphisms 
		$$ T_\sD \relbar\joinrel\twoheadrightarrow \omega_C \relbar\joinrel\twoheadrightarrow \mathcal{O}_{L_1}(-3) \oplus \mathcal{O}_{L_2}(-2). $$
		which can only exist if $L_1$ and $L_2$ are jumping lines for $T_\sD$ of orders $3$ and $2$, respectively.
	\end{remark}
	
	Recall that isomorphism classes of codimension one distributions $\sD$ of degree $d$ on $\p3$ whose tangent sheaves have Chern classes $c_2(T_\sD) = c_2$ and $c_3(T_\sD))=c_3$ are parameterized by a quasi-projective variety denoted by $\mathcal{D}(d,c_2,c_3)$, see \cite{CCJ1}. Some of these moduli spaces have been explicitly described in \cite[Section 11]{CCJ1} and \cite{CJMdeg1}, using different techniques. The results obtained in this section provide us with a detailed description of $\mathcal{D}(2,3,0)$.
	
	\begin{proposition}\label{prop:moduli30}
		The moduli space $\mathcal{D}(2,3,0)$ of codimension one distributions $\sD$ of degree 2 such that $c_2(T_\sD)=3$ and $c_3(T_\sD)=0$ is irreducible of dimension 21.
	\end{proposition}
	\begin{proof}
		Let $\calb(3)$ denote the moduli space of stable rank 2 locally free sheaves $E$ with $c_1(E)=0$ and $c_2(E)=3$. If $[\sD]\in\mathcal{D}(2,3,0)$, then Theorem \ref{thm:stability} implies that $[T_\sD]\in\calb(3)$, thus \cite[Lemma 2.5]{CCJ1} yields a forgetful morphism $\varpi\colon \mathcal{D}(2,3,0)\to\calb(3)$.
		
		Let $\cali^1(3)$ denote the space of instanton bundles with a unique jumping line of order 3, which was shown to be an irreducible quasi-projective variety of dimension 20 in \cite[1.4]{GS}, and let $\calj$ denote the open subset of $\cali^1(3)$ consisting of instanton bundles corresponding to a smooth rational quintic as in \cite[2.2]{GS}. Following the proof of Theorem \ref{thm:instanton3}, we note that $\calj\subseteq\im(\varpi)\subseteq\cali^1(3)$.
		
		Furthermore, the fibre of $\varpi$ over $[E]\in\im(\varpi)$ is precisely the set of monomorphisms $\phi\colon E\to\tp3$ whose cokernel is torsion free, so it is an open subset of $\mathbb{P}\Hom(E,\tp3)$. However, as we saw in the proof of Theorem \ref{thm:instanton3}, if $[E]\in\calj$, then every $\phi\in\Hom(E,\tp3)$ is injective and has torsion free cokernel, thus in fact $\varpi^{-1}([E])=\mathbb{P}\Hom(E,\tp3)\simeq\p1$ whenever $[E]\in\calj$. Since $\hom(E,\tp3) = 2$ for every $[E] \in \cali^1(3)$ it follows that $\dim\varpi^{-1}([E])=1$ for every $E\in\im(\varpi)$. The conclusion is then an immediate consequence of the Theorem on the Dimension of Fibres.
	\end{proof}

	
	\section{Distributions with \texorpdfstring{$c_2(T_\sD)=4 $}{c2 =4}}\label{sect:c2=4}
	
	Let $\sD$ be a degree two distribution such that $c_2(T_\sD)=4$. Then Theorem \ref{thm:stability} implies that $T_\sD$ is stable and, due to \eqref{eq:bound c3}, $c_3(T_\sD) \in \{0,2,4,6,8,10,12,14\}$. In fact, we can prove that $c_3(T_\sD) \in \{6,8,10\}$. 
	
	First note that $C = \sing_1(\sD)$ is a curve of degree $2$ and, in particular, $p_a(C) \leq 0$, see \cite[Corollary 1.6]{N}. Thus \eqref{eq:chern sing1 deg2} implies that
	\[
	c_3(T_\sD) = 10 + 2p_a(C) \leq 10.
	\]
	On the other hand, \cite[Corollary 1.6]{N} also says that any degree two curve of genus less than $-1$ is a double line. Due to Corollary \ref{cor:2line}, we must have $p_a(C) \geq -2$, whence
	\[
	c_3(T_\sD) \geq 6.
	\]
	
	Before we show the existence of such distributions, we make a few remarks on the cohomology of the possible tangent sheaves. \begin{lemma}\label{lem:h0(T_D(1))=0,c2=4}
		If $T_\sD$ is the tangent sheaf of a distribution $\sD$ of degree 2 such that $c_2(T_\sD)= 4$ and $c_3(T_\sD))=2l$, then
		\[
		h^0(T_\sD(1)) = \begin{cases}
			0 , & l =3, 4\\
			1, & l=5
		\end{cases}.
		\]
	\end{lemma}
	
	\begin{proof}
		From Lemma \ref{lem:h2=0}, we have that $h^0(T_\sD(1)) = h^0(\omega_C(1)) = h^1(\mathcal{O}_C(-1))$, where $C = \sing_1(\sD)$; we will compute $h^1(\mathcal{O}_C(-1))$. From \eqref{eq:chern sing1 deg2} we get $\deg(C) = 2$ and $p_a(C) = l-5$.
		
		For $l\leq 4$, it follows that
		\[
		0 \longrightarrow \mathcal{O}_{L_1}(3-l) \longrightarrow \mathcal{O}_C(-1) \longrightarrow \mathcal{O}_{L_2}(-1) \longrightarrow 0
		\]
		where $L_1$ and $L_2$ are reduced lines and $L_1\neq L_2$ only if $l=4$, see \cite{N}. It is then clear that $h^1(\mathcal{O}_C(-1)) =0$. 
		
		For $l=5$ the curve $C$ is a conic (not necessarily irreducible nor reduced) whence
		\[
		0 \longrightarrow \op3(-4) \longrightarrow \op3(-3) \oplus \op3(-2) \longrightarrow \mathscr{I}_C(-1) \longrightarrow 0
		\]
		and it follows that $h^1(\mathcal{O}_C(-1))= h^2(\mathscr{I}_C(-1)) = 1$. 
	\end{proof}

	
	\subsection{Case \texorpdfstring{$ c_3(T_\sD) =6$}{c3 = 6}} \label{sec:4-6}
	In this case we have that $\sing(\sD)$ is the union of a double line $C$ of genus $-2$ with $6$ points. We will show that these distributions exist with the following example.
	
	\begin{example}\label{ex:c2=4 c3=6}
		Consider the double structure $C$ on the line $\{x=y=0\}$ given by the ideal
		\[
		I_C = (x^2,xy, y^2, x(z^2-w^2)-yzw)
		\]
		and fix the points $(1:0:0:0)$, $(0:1:0:0)$, $(1:1:1:1)$ and $(-1:1:-1:1)$. From them we can compute the following $1$-form
		\begin{align*}
			\omega = &(-xy^2+xyz+y^2z+2xz^2+2xyw-y^2w-2yzw-2xw^2)dx\\ &+(x^2y+x^2z+xyz+xz^2-2x^2w-yzw-xw^2)dy\\ &+(-2x^2y-2xy^2-2x^2z-xyz+x^2w+4xyw+2y^2w)dz\\ &+(xy^2-x^2z-2xyz-y^2z+2x^2w+xyw)dw,
		\end{align*}
		which is singular at the prescribed scheme and also at $(38:-19:14:-4)$ and $(9:-27:17:13)$.
	\end{example}

	Let $\{k_1, k_2, k_3 ,k_4\}$ be the spectrum of $T_\sD$ for a distribution $\sD$ with $c_2(T_\sD)= 4$ and $c_3(T_\sD)= 6$.
	In Lemma \ref{lem:spec46} below, we will show that $h^2(T_\sD(-1))=0$ hence Remark \ref{rem:properties of spectrum} implies that $k_i \geq -1$. Since $k_1+k_2+k_3+k_4 = -3$, the only possible spectrum is
	\[
	\{-1,-1,-1,0\}.
	\]

	\begin{lemma}\label{lem:spec46}
		If $\cald$ is a codimension one distribution on $\p3$ of degree $2$ such that $(c_2(T_\sD),c_3(T_\sD))=(4,6)$, then $h^2(T_\sD(-1))=0$.
	\end{lemma}
	
	\begin{proof}
		From Lemma \ref{lem:h2=0} and the Hirzebruch-Riemann-Roch Theorem, we know that $h^2(T_\sD(-2))= 3$ and $h^2(T_\sD(-1)) \leq 1$. Let us assume, by contradiction, that $h^2(T_\sD(-1)) = 1$. Let $F:= T_\sD$ to simplify notation.
		
		For any hyperplane $H = \{h =0 \} \subset \p3$ we have the exact sequence
		\[
		0 \longrightarrow F(-2)  \stackrel{\cdot h}{\longrightarrow} F(-1)  \longrightarrow F_H(-1)  \longrightarrow 0.
		\]
		Taking the long exact sequence of cohomology and dualizing it we can extract the following:
		\[
		H^2(F_H(-1))^\vee \longrightarrow H^2(F(-1))^\vee \stackrel{\cdot h}{\longrightarrow} H^2(F(-2))^\vee.
		\]
		Thus $h^2(F_H(-1)) \neq 0$ if and only if $h$ annihilates $H^2(F(-1))^\vee$. On the other hand, $h^2(F_H(-1)) \neq 0$ if and only if $H$ is an unstable plane of order $2$ for $F$, see \cite[Section 9]{H2}. To see that such unstable planes exist, we take the multiplication map
		\[
		H^0(\op3(1))\otimes H^2(F(-1))^\vee \longrightarrow H^2(F(-2))^\vee.
		\]
		Since $h^2(F(-1)) = 1$ and $h^2(F(-2))= 3$, the kernel of the map above is nontrivial. 
		
		Let $H$ be a an unstable  plane of order $2$ for $F$. We then have a reduction step, see \cite[Proposition 9.1]{H2}, 
		\[
		0 \longrightarrow F' \longrightarrow F  \longrightarrow \mathscr{I}_{W/H}(-2)  \longrightarrow 0,
		\]
		where $F'$ is a rank $2$ reflexive sheaf and $W$ is a 0-dimensional subscheme of $H$. Thus  $c_1(F') = -1$ and $c_2(F') = 2$. Moreover, $h^0(F') = h^0(F) = 0$ hence $F'$ is stable and $c_3(F')\in \{0,2,4\}$. 
		
		Note that $h^0(F(1)) = 0$, due to Lemma \ref{lem:h0(T_D(1))=0,c2=4}. However, we get a contradiction since $h^0(F'(1)) >0$. 
		This last fact is due to \cite{HSols} if $c_3(F') = 0$, \cite[Table 2.3.1]{Ch} if $c_3(F') = 2$ and \cite[Lemma 9.6]{H2} if $c_3(F') = 4$.
	\end{proof}

	
	\subsection{Case \texorpdfstring{$ c_3(T_\sD) =8$}{c3 = 8}}
	\label{sec:4-8}
	
	For $\sD$ a distribution of degree two with $c_2(T_\sD) = 4$ and $c_3(T_\sD) =8$ we have that $\sing(\sD)$ is a curve $C$ of degree $2$ and genus $-1$ plus $8$ points. In particular $C$ is either a pair of skew lines or a double line. We will show that such distributions exist using an auxiliary foliation.
	
	Let $N$ be a null correlation bundle, and let $\sigma\in H^0(N(1))$ such that $C:=(\sigma)_0$ consists of the union of two disjoint lines. Let $\sG$ be a generic foliation by curves of degree $3$ given by the exact sequence
	\[
	\sG ~:~ 0 \longrightarrow N(-3) \longrightarrow \Omega^1_{\p3} \longrightarrow \mathscr{I}_W(2) \longrightarrow 0,
	\]
	so that $W$ is irreducible; such foliations exist by \cite[Proposition 8.3]{CJM}. As $\deg(\sG) =3$ and $W$ is irreducible, we may apply Proposition \ref{prop: const dist folbc} to conclude that the distribution $\sD$ induced by $(\sG,\sigma)$ satisfies $\sing_1(\sD) = C$. Moreover, it follows from $\eqref{eq:chern dist from folbc deg2}$ that $c_2(T_\sD)=4$ and $c_3(T_\sD)=8$.
	
	The next step is to determine the possible spectra of the tangent sheaf $T_\sD$, and we proceed as in the previous subsection. Let $\{k_1, k_2, k_3 ,k_4\}$ be the spectrum of $T_\sD$. In Lemma \ref{lem:spec48} below, we show that $h^2(T_\sD(-1))=0$ hence $k_i \geq -1$. As $k_1+k_2+k_3+k_4 = -4$ we see that the only possible spectrum is 
	\[
	\{-1,-1,-1,-1\}
	\]

	\begin{lemma}\label{lem:spec48}
		If $\cald$ is a codimension one distribution on $\p3$ of degree $2$ such that $(c_2(T_\sD),c_3(T_\sD))=(4,8)$, then 
		\[
		0 \longrightarrow \tp3(-4)^{\oplus 2} \longrightarrow \op3(-2)^{\oplus 8} \longrightarrow T_\sD \longrightarrow 0;
		\]
		in particular, $h^2(T_\sD(-1))=0$.
	\end{lemma}
	
	\begin{proof}
		From Lemma \ref{lem:h2=0} and noting that $C$ is a curve of degree $2$ and genus $-1$, we get that $h^1(T_\sD(1)) = h^2(T_\sD) = h^3(T_\sD(-1)) = 0$ and $h^2(T_\sD(-1)) \leq 1$. In particular, $T_\sD(2)$ is globally generated due to the Castelnuovo--Mumford criterion;  in addition, $h^0(T_\sD(2))=\chi(T_\sD(2))=8$. Let $E(2)$ be the kernel of the evaluation of global sections, then
		\begin{equation} \label{sqc e td}
			0 \longrightarrow E \longrightarrow \op2(-2)^{\oplus 8} \longrightarrow T_\sD\longrightarrow 0.
		\end{equation}
		Since $T_\sD$ is reflexive, we have that $\inext^p(E,\op3)\simeq\inext^{p+1}(T_\sD,\op3))=0$ for $p\ge1$, thus $E$ must be locally free; computing its cohomologies we also see that $E(3)$ is globally generated and, due to its Chern classes, we get
		\[
		0 \longrightarrow \op3(-4)^{\oplus 2} \longrightarrow \op2(-3)^{\oplus 8} \longrightarrow E \longrightarrow 0.
		\]
		Therefore $T_\sD$ has a resolution 
		\begin{equation}\label{eq:res48}
			0 \longrightarrow \op3(-4)^{\oplus 2} \stackrel{A}{\longrightarrow} \op2(-3)^{\oplus 8} \stackrel{B}{\longrightarrow} \op2(-2)^{\oplus 8} \longrightarrow T_\sD\longrightarrow 0
		\end{equation}
		where $A$ and $B$ are matrices of linear homogeneous polynomials. 
		
		Note that the above resolution allows us to compute
		\[
		H^2(T_\sD(-1)) = \Ext^1(T_\sD, \op3(-3)) = \ker \left( H^0(\op3)^{\oplus 8} \stackrel{ A^T \cdot}{\longrightarrow} H^0(\op3(1))^{\oplus 2} \right)
		\]
		where $A^T$ is the transpose of $A$. We may write $A = (a_{ij})$ with $a_{ij} = \sum_{k=0}^3 a_{ij}^kx_k$ then we define a  pencil of $4\times 8$ matrices $C(t) = (c_{ij}(t))$ where $c_{ij}(t) = a^i_{j1} + a^i_{j2}t$. We can assume that $C(t)$ is in Kronecker normal form, see \cite[Chapter XII]{GANT}, so that we only need to analyse few possibilities for $A$, up to linear change of coordinates and left action of ${\rm GL}(8, \mathbb{C})$. In fact, since $h^2(T_\sD(-1)) \leq 1$, we have only two possibilities:
		\[
		A^T = \begin{bmatrix}
			x_0 & 0 & x_1 & 0 & x_2 & 0 & x_3 &0 \\ 0 & x_0& 0  & x_1 & 0 & x_2 & 0 & x_3
		\end{bmatrix} , \begin{bmatrix}
			x_0 & 0 & x_1 & 0 & x_2 & x_3 & 0 &0 \\ 0 & x_0& 0  & x_1 & 0 & x_2 & x_3 & 0
		\end{bmatrix}.
		\]
		
		In the first case we rearrange the columns 
		\[
		A^T = \begin{bmatrix}
			x_0 & x_1  & x_2  & x_3 &  0& 0  & 0 &0 \\ 0 & 0& 0  & 0 & x_0 & x_1  & x_2  & x_3 
		\end{bmatrix}
		\]
		so that it is clear that 
		\[
		E = \coker \left( A\colon \op3(-4)^{\oplus 2} \longrightarrow \op2(-3)^{\oplus 8}  \right)  = \tp3(-4)^{\oplus 2}.
		\]
		In particular, $h^2(T_\sD(-1))= 2h^3(\tp3(-5)) = 0$. 
		
		Next we show that the second case cannot happen. Since the columns of $B^T$ are linear syzygies for the rows of $A^T$, we can compute them explicitly and it follows that, for the second case, up to the right action of ${\rm GL}(8, \mathbb{C})$, 
		\[
		B = \begin{bmatrix}
			-y & 0 & x & 0 & 0 & 0 & 0 & \ast \\
			0 & -y & 0 & x & 0 &  0 & 0 & \ast \\
			0 & 0 & -z & 0 & y & 0 & 0 & \ast \\
			-z & 0 & 0 & 0 & x & 0 & 0 & \ast \\
			0 & 0 & -w & -z & 0 & y & 0 & \ast \\
			-w & -z & 0 & 0 & 0 & x & 0 & \ast \\
			0 & 0 & 0 & -w & 0 & 0 & y & \ast \\
			0 & -w & 0 & 0 & 0 & 0 & x & \ast \\
		\end{bmatrix},
		\]
		where the $(*)$ are unknown linear polynomials. Note that $B$ has generic rank $6$ and that $B$ has rank $\leq 5$ along the line $\{x=y=0\}$. Thus the cokernel of $B \colon \op2(-3)^{\oplus 8} \rightarrow \op2(-2)^{\oplus 8}$ is not locally free around any point of the line $\{x = y = 0\}$; in particular, $\coker B$ cannot be reflexive, which $T_\sD$ is. Therefore the second case cannot occur. 
		
	\end{proof}

	
	\subsection{Case \texorpdfstring{$ c_3(T_\sD) =10$}{c3 = 10}}
	\label{sec:4-10}
	
	A distribution $\sD$ in this case is singular at a conic plus $10$ points. We may show that such distributions exist using an auxiliary foliation by curves. 
	
	Start with a generic foliation by curves of degree 2 of the form
	\begin{equation}\label{sg 4-10}
		\sG ~:~ 0 \longrightarrow \op3(-2)\oplus\op3(-3) \longrightarrow \Omega^1_{\p3} \longrightarrow \mathscr{I}_W(1) \longrightarrow 0,
	\end{equation}
	with $W$ being a smooth and irreducible curve of degree 5 and genus 1. Since $N_\sG^*(4) = \op3(2)\oplus\op3(1)$, we can find $\sigma\in H^0(N_\sG^*(4))$ such that $C = (\sigma)_0$ is a conic disjoint from $W$. Therefore, Proposition \ref{prop: const dist folbc} guarantees that the induced codimension one distribution $\sD$ is such that $\sing_1(\sD)=C$ provided $\eta\in H^0(\omega_Z(2))$ is non zero; the formulas in display \eqref{eq:chern sing1 deg2} yield $c_2(T_\sD) = 4 $ and $c_3(T_\sD))=10$. We only need to show that $\eta \neq 0$.
	
	Assume by contradiction that $\eta=0$, so that the long exact sequence in display \eqref{long sqc} breaks into two parts
	\begin{align*}
		0 \longrightarrow \op3(-1) \longrightarrow F^* \longrightarrow \op3(1) \longrightarrow 0; \\
		0 \longrightarrow \omega_W(3) \longrightarrow \mathcal{O}_Z(4) \longrightarrow \omega_C(5)\longrightarrow 0.
	\end{align*}
	Using that $\omega_W = \mathcal{O}_W$ and $\omega_C = \mathcal{O}_C(-1)$ we conclude that
	\[
	0 \longrightarrow \mathcal{O}_W(-1) \longrightarrow \mathcal{O}_Z \longrightarrow \mathcal{O}_C\longrightarrow 0
	\]
	hence $Z = W\cup C$ and $W\cap C \neq \emptyset$, which is absurd. 
	The next step is to determine the spectrum of the tangent sheaf. Let $\sD$ be a distribution of degree $2$ satisfying $c_2(T_\sD) = 4 $ and $c_3(T_\sD))=10$ and let $\{k_1, k_2,k_3,k_4\}$ be the spectrum of $T_\sD$. First we claim that 
	$h^1(T_\sD(-1))=0$, so that $k_1 \leq -1$. It follows that $k_i \geq -2$; indeed as $k_1+ k_2+k_3+k_4 = -5$ we could not have $\{-3,-2,-1\}$ contained in the spectrum. Therefore, the only possible spectrum is
	\[
	\{-2,-1,-1,-1\}.
	\]
	Now we prove our claim.
	
	\begin{lemma} 
		If $T_\sD$ is the tangent sheaf of a codimension one distribution $\sD$ of degree 2 with $(c_2(T_\sD),c_3(T_\sD))=(4,10)$, then $h^1(T_\sD(-1))=0$. 
	\end{lemma}
	
	\begin{proof}
		By Lemma \ref{lem:h0(T_D(1))=0,c2=4}, $h^0(T_\sD(1))=1$, so let $\sigma\in H^0(T_\sD(1))$ be a nontrivial section and let $X:=(\sigma)_0$ be its vanishing locus. We have then the following commutative diagram:
		\[
		\xymatrix{
			& 0 \ar[d] & 0\ar[d] & & \\
			& \op3(-1) \ar[d]^{\sigma} \ar@{=}[r] & \op3(-1) \ar[d] & & \\
			0\ar[r] & T_\sD \ar[d]\ar[r]^{\phi} & \tp3 \ar[d]\ar[r] & \mathscr{I}_Z(4)\ar[r]\ar@{=}[d] & 0 \\
			0\ar[r] & \mathscr{I}_X(1) \ar[d]\ar[r] & G \ar[d]\ar[r] & \mathscr{I}_Z(4)\ar[r] & 0 \\
			& 0 & 0 & &
		}
		\]
		where $Z= \sing(\sD)$ is a conic $C$ plus $10$ points and $G$ is the cokernel of $\phi\circ\sigma$. From the first column from the left we have $h^1(T_\sD(-1)) = h^1(\mathscr{I}_X)$; we will compute $h^1(\mathscr{I}_X)$.
		
		Dualizing the second column from the left we get a foliation by curves 
		\begin{equation}\label{seq:folindc2=4,c3=10}
			0 \longrightarrow G^\ast \longrightarrow \Omega^1_\p3 \longrightarrow \mathscr{I}_W(1) \longrightarrow 0
		\end{equation}
		where $W\supset X$; and dualizing the bottom row we get 
		\[
		0 \longrightarrow \op3(-4) \longrightarrow G^* \longrightarrow \op3(-1) \stackrel{\zeta}{\longrightarrow} \omega_C \longrightarrow \cdots
		\]
		where $\zeta \in H^0(\omega_C(1))$. Note that as $C$ is a conic, $\omega_C = \mathcal{O}_C(-1)$ and also note that $\zeta$ is surjective; otherwise $\zeta= 0$ and $G^\ast$ would not admit an injective map to $\Omega^1_\p3$.
		
		In particular, we get 
		\[
		0 \longrightarrow \op3(-4) \longrightarrow G^*\longrightarrow \mathscr{I}_C(-1) \longrightarrow 0
		\]
		and the only possibility is $G^\ast = \op3(-2) \oplus \op3(-3)$. It follows that $W$ has pure dimension one and by degree reasons $X=W$. Thus \eqref{seq:folindc2=4,c3=10} gives as a resolution for $\mathscr{I}_X$ from which we get
		\[
		h^1(T_\sD(-1)) = h^1(\mathscr{I}_X) = 0.
		\]
	\end{proof}


	\section{Distributions with \texorpdfstring{$c_2(T_\sD)=5 $}{c2 =5}}\label{sect:c2=5}
	
	If $\sD$ is a distribution of degree two such that $c_2(T_\sD)=5$, then $\sing_1(\sD)$ is a curve of degree $1$, i.e., a reduced line. Thus $c_3(T_\sD)= 14$, owing to \eqref{eq:chern sing1 deg2}. These distributions do exist and, in fact, we can prove a more general result.
	
	\begin{proposition}\label{prop:exist c2=5} 
		For each $d\ge2$ and any line $L\subset \mathbb{P}^3$, there exists a codimension one distribution $\sD$ of degree $d$ whose singular scheme consists of $L$ plus $d^3+2d^2-d$ points, so that $(c_2(T_\sD),c_3(T_\sD))=(d^2+1,d^3+2d^2-d)$.
	\end{proposition}
	
	In particular, taking $d=2$ in the statement above implies that there exists a codimension one distribution $\sD$ of degree 2 such that $(c_2(T_\sD),c_3(T_\sD))=(5,14)$.
	
	\begin{proof}
		The starting point is a foliation by curves of the form
		\begin{equation}\label{seq:folbc5-14}
			\sG ~:~ 0 \longrightarrow \op3(-d-1)^{\oplus2} \stackrel{\phi}{\longrightarrow} \Omega^1_{\p3} \longrightarrow \mathscr{I}_W(2d-2) \longrightarrow 0; 
		\end{equation}
		note that $\deg(\sG)=2d-1\ge3$, so $W$ must be connected by \cite[Proposition 4.3]{CJM}. By taking a generic morphism $\phi$, Ottaviani's Bertini-type Theorem \cite[
		Teorema 2.8]{O} implies that we can assume $W$ to be smooth, so $W$ is an irreducible curve.
		
		Since $N_\sG^*=\op3(-d-1)^{\oplus2}$, we can find $\sigma\in H^0(N_\sG^*(d+2))$ so that $(\sigma)_0=L$, for any given line $L$. We then apply Proposition \ref{prop: const dist folbc} to the pair $(\sG,\sigma)$ we just described. Note that for $d\geq 2$ we have $\deg(\sG)\geq d+1$ and Remark \ref{rem:eta=0} implies $\eta \neq 0$; since $W$ is irreducible, $\dim \coker \eta = 0$. 
		
		The induced distribution $\sD$ will have degree $d$ and $\sing_1(\sD)=L$; the Chern classes for $T_\sD$ follow directly from \eqref{eq:chern dist from folbc}.
	\end{proof}
	
	\begin{remark}
		There exist distributions as in the statement of Proposition \ref{prop:exist c2=5} also for $d=0$ or $1$, but the proof above does not apply to these cases: when $d=0$, there is no morphism $\op3(-1)^{\oplus2} \rightarrow \Omega^1_{\p3}$; when $d=1$, the singular scheme $W$ is not connected. 
	\end{remark}
	
	To conclude this section we compute the possible spectra for the tangent sheaf of a codimension one distribution of degree 2 with $c_2(T_\sD)=5$. Let $\{k_1, \dots, k_5\}$ be the spectrum of $T_\sD$; note that $k_1+\cdots+k_5=-7$ and $k_i \geq -3$ due to Lemma \ref{lem:bound spectrum}. We can prove, see Lemma \ref{lem:spec514} below, that $h^1(T_\sD(-1)) = 0$. Hence $k_i\leq -1$ for every $i$ and the only possible spectrum is 
	\[
	\{-2,-2,-1,-1,-1\}.
	\]

	\begin{lemma}\label{lem:spec514}
		If $\cald$ is a codimension one distribution on $\p3$ of degree $2$ such that $(c_2(T_\sD),c_3(T_\sD))=(5,14)$, then $h^1(T_\sD(-1))=0$.
	\end{lemma}
	
	\begin{proof} 
		Due to Lemma \ref{lem:h2=0}, $h^0(T_\sD(1)) = h^1(T_\sD(2)) = h^2(T_\sD(2)) =h^3(T_\sD(2)) = 0$; hence $h^0(T_\sD(2)) = 7$, by Hirzebruch--Riemann--Roch Theorem. Then a general section $\sigma \in H^0(T_\sD(2))$ induces an exact sequence
		\[
		0 \longrightarrow \op3(-4) \longrightarrow F(-2) \stackrel{\sigma^{\vee}}{\longrightarrow} \mathscr{I}_Y \longrightarrow 0,
		\]
		where $Y$ is a Cohen--Macaulay curve of degree $9$ and (arithmetic) genus $8$. Moreover, 
		\begin{itemize}
			\item $h^0(\mathscr{I}_Y(p))= 0$ for $p\leq 3$;
			\item $h^2(\mathscr{I}_Y(p))= 0$ for $p\geq 3$;
			\item $h^1(\mathscr{I}_Y(p)) = 0$ for $p\not \in \{1,2\}$, $h^1(\mathscr{I}_Y(1))\leq 1$ and $1 \leq h^1(\mathscr{I}_Y(2))\leq 2$.
		\end{itemize}
		
		Take a hyperplane section $\Gamma = Y \cap H$. As $Y$ has pure dimension $1$, for any $H = \{h=0\}$ such that $\dim \Gamma = 0$ we have a short exact sequence
		\begin{equation}\label{seq:hsecY98}
			0 \longrightarrow \mathscr{I}_Y(-1) \stackrel{\cdot h}{\longrightarrow} \mathscr{I}_Y \longrightarrow \mathscr{I}_{\Gamma/H} \longrightarrow 0,
		\end{equation}
		where $\mathscr{I}_{\Gamma/H}\subset \mathcal{O}_H$ is the ideal sheaf of $\Gamma$ in $H \simeq \p2$. 
		
		First we note that $h^0(\mathscr{I}_{\Gamma/H}(2)) =  0$. Indeed, $h^0(\mathscr{I}_Y(2)) = 0$ implies that 
		\[
		h^0(\mathscr{I}_{\Gamma/H}(2)) \leq h^1(\mathscr{I}_Y(1)) \leq 1
		\]
		and $h^0(\mathscr{I}_{\Gamma/H}(2)) = 1$ would imply $h^0(\mathscr{I}_{\Gamma/H}(3)) \geq 3$. However, $h^0(\mathscr{I}_{\Gamma/H}(3))\leq 2$ since $h^0(\mathscr{I}_Y(3)) = h^1(\mathscr{I}_Y(3)) =0$.
		
		We will now show that $h^1(\mathscr{I}_Y(1)) = 0$. Let us assume, by contradiction, that $h^1(\mathscr{I}_Y(1)) >0$, hence $h^1(\mathscr{I}_Y(1)) = 1$. From the sequence \eqref{seq:hsecY98}, we have that 
		\[
		H^0(\mathscr{I}_{\Gamma/H}(2)) = 0 \longrightarrow H^1(\mathscr{I}_Y(1))  \stackrel{\cdot h}{\longrightarrow}  H^1(\mathscr{I}_Y(2)),
		\]
		hence the multiplication with $h$ is injective whenever $\{h=0\} \cap Y$ is $0$-dimensional. Fix a generator $v\in H^1(\mathscr{I}_Y(1))\setminus \{0\}$ and consider the map 
		\[
		\phi \colon H^0(\op3(1))  \rightarrow  H^1(\mathscr{I}_Y(2)) ; \quad \phi(h) = hv .
		\]
		Note that $\dim \ker \phi = 3$, since  $\phi \neq 0$ and $h^1(\mathscr{I}_Y(2)) = 1$. 
		
		Let $W \subset \mathbb{P}H^0(\op3(1))$ be the the Zariski closed subset defined by 
		\[
		W = \{ H \in \mathbb{P}H^0(\op3(1)) \mid \dim H \cap Y = 1 \}.
		\]
		In particular we have that $\mathbb{P} (\ker \phi) \subset W$, hence $\dim W \geq 2$. We will see that this is absurd, concluding the proof.
		
		Given a hyperplane $H$, we have $H \cap Y$ has dimension $1$ if and only if $H$ contains some irreducible component of $Y_{red}$. Then let $Y_{red} = Y_1 \cup Y_2 \cup \dots \cup Y_r$ be a decomposition into irreducible components and define 
		\[
		W_i  = \{ H \in \mathbb{P}H^0(\op3(1)) \mid Y_i \subset H \}.
		\]
		Note that $W = \bigcup_{i=1}^r W_i$ and $W_i\neq \emptyset$ if and only if $Y_i$ is a plane curve, and $\dim W_i = 1$ only if $Y_i$ is a line. Therefore $\dim W \leq 1$. 
	\end{proof}


	\section{Distributions with \texorpdfstring{$c_2(T_\sD)=6 $}{c2 =6}}\label{sect:c2=6}
	
	The first line of Table \ref{deg 2 table} refers to the distributions with $c_2(T_\sD)=6$. These are the \emph{generic distributions} of degree 2, that is, those distributions whose singular scheme is 0-dimensional. In particular, $c_3(T_\sD) = 20$. The subset of $\mathbb{P}H^0(\op3(4))$ parameterizing isomorphism classes of generic distributions of degree two is Zariski open. One can easily find an example of such a distribution.
	
	\begin{example}
		Let $\sD$ be the distribution given by the following $1$-form, inspired by Jouanolou's examples \cite[p. 160]{Jou}, 
		\[
		\omega = (x^2y-w^3)dx + (y^2z-x^3)dy + ( z^2w-y^3)dz + (w^2x-z^3)dw.
		\]
		Its singular scheme is the reduced set of points $\{(\xi: \xi^{-2} : \xi^7 : 1) \mid \xi^{20} = 1 \}$.
	\end{example}
	
	Generic distributions of degree two are completely determined by its singular scheme, see \cite[Corollary 4.7]{CJMdeg1}; and they are never integrable, due to the dimension of the singular scheme.

	It only remains for us to determine the spectrum of the tangent sheaf. Let $\sD$ be a generic distribution and let $\{k_1, \dots , k_6\}$ be the spectrum of $T_\sD$. Since $\sing(\sD)$ has codimension three, $\inext^1(\mathscr{I}_Z, \op3) = 0$ and dualizing the exact sequence in display \eqref{seq:dist-p3} we obtain
	\[
	0 \longrightarrow \op3(-4) \longrightarrow \Omega^1_{\p3} \longrightarrow T_\sD \longrightarrow 0
	\]
	from which it follows that $h^1(T_\sD(-1)) = 0$ and $h^2(T_\sD(1))= 0$ and $h^2(T_\sD)= 1$. Due to \eqref{eq:def spectrum}, we have $-3 \leq k_i \leq -1$, for every $i$, and $-3$ occurs once in the spectrum. Since $k_1 + \cdots + k_6 = -10$, an easy computation of integer partitions shows that the spectrum is $\{-3,-2,-2,-1,-1,-1\}$.

	

\begin{thebibliography}{10}
		
		\bibitem{Am}
		M.~{Amasaki}.
		\newblock {On the structure of arithmetically Buchsbaum curves in \(P^ 3_ k\)}.
		\newblock {\em {Publ. Res. Inst. Math. Sci.}}, 20:793--837, 1984.
		\newblock \href {http://dx.doi.org/10.2977/prims/1195181111}
		{\path{doi:10.2977/prims/1195181111}}.
		
		\bibitem{CCJ1}
		O.~Calvo-Andrade, M.~Corrêa, and M.~Jardim.
		\newblock {Codimension One Holomorphic Distributions on the Projective
			Three-space}.
		\newblock {\em Int. Math. Res. Not.}, 23:9011--9074, 2020.
		\newblock \href {http://dx.doi.org/10.1093/imrn/rny251}
		{\path{doi:10.1093/imrn/rny251}}.
		
		
		\bibitem{CaJS}
		A.~Cavalcante, M.~Jardim, and D.~Santiago.
		\newblock Foliations by curves on threefolds, 2021.
		\newblock To appear in {\em Math. Nach.}
		\newblock \href {http://arxiv.org/abs/2101.06244} {\path{arXiv:2101.06244}}.
		
		\bibitem{CLN}
		D.~Cerveau and A.~Lins~Neto.
		\newblock Irreducible components of the space of holomorphic foliations of
		degree two in .{$\mathbb C{\rm P}(n)$}, {$n\geq 3$}
		\newblock {\em Ann. of Math. (2)}, 143(3):577--612, 1996.
		\newblock \href {http://dx.doi.org/10.2307/2118537}
		{\path{doi:10.2307/2118537}}.
		
		\bibitem{Ch}
		M.-C. Chang.
		\newblock Stable rank {$2$} reflexive sheaves on {${\bf P}^{3}$} with small
		{$c_{2}$} and applications.
		\newblock {\em Trans. Amer. Math. Soc.}, 284(1):57--89, 1984.
		\newblock \href {http://dx.doi.org/10.2307/1999274}
		{\path{doi:10.2307/1999274}}.
		
		\bibitem{CJV}
		M.~Corr\^{e}a, Jr., M.~Jardim, and R.~V. Martins.
		\newblock On the singular scheme of split foliations.
		\newblock {\em Indiana Univ. Math. J.}, 64(5):1359--1381, 2015.
		\newblock \href {http://dx.doi.org/10.1512/iumj.2015.64.5672}
		{\path{doi:10.1512/iumj.2015.64.5672}}.
		
		\bibitem{CJM}
		M.~Corrêa, M.~Jardim, and S.~Marchesi.
		\newblock Classification of foliations by curves of low degree on the
		three-dimensional projective space, 2020.
		\newblock \href {http://arxiv.org/abs/1909.06590} {\path{arXiv:1909.06590}}.
		
		\bibitem{CJMdeg1}
		M.~Corrêa, M.~Jardim, and A.~Muniz.
		\newblock Moduli of distributions via singular schemes, 2020.
		\newblock \href {http://arxiv.org/abs/2010.02382} {\path{arXiv:2010.02382}}.
		
		\bibitem{FassarellaThesis}
		T.~Fassarella.
		\newblock {\em Sobre a Aplicação de Gauss de Folheações Holomorfas em
			Espaços Projetivos}.
		\newblock PhD thesis, IMPA, Rio de Janeiro, 2008.
		\newblock URL:
		\url{https://impa.br/wp-content/uploads/2017/08/tese_dout_thiago_fassarella_amaral.pdf}.
		
		\bibitem{GANT}
		F.~R. Gantmacher.
		\newblock {\em The theory of matrices. {V}ols. 1, 2}.
		\newblock Chelsea Publishing Co., New York, 1959.
		\newblock Translated by K. A. Hirsch.
		
		\bibitem{M2}
		D.~R. Grayson and M.~E. Stillman.
		\newblock Macaulay2, a software system for research in algebraic geometry.
		\newblock Available at \url{http://www.math.uiuc.edu/Macaulay2/}.
		
		\bibitem{GS}
		L.~{Gruson} and M.~{Skiti}.
		\newblock {3-instantons et r\'eseaux de quadriques}.
		\newblock {\em {Math. Ann.}}, 298(2):253--273, 1994.
		\newblock \href {http://dx.doi.org/10.1007/BF01459736}
		{\path{doi:10.1007/BF01459736}}.
		
		\bibitem{HART:AG}
		R.~Hartshorne.
		\newblock {\em Algebraic geometry}.
		\newblock Springer-Verlag, New York-Heidelberg, 1977.
		\newblock Graduate Texts in Mathematics, No. 52.
		
		\bibitem{H3}
		R.~{Hartshorne}.
		\newblock {Stable vector bundles of rank 2 on \(P^ 3\)}.
		\newblock {\em {Math. Ann.}}, 238:229--280, 1978.
		\newblock \href {http://dx.doi.org/10.1007/BF01420250}
		{\path{doi:10.1007/BF01420250}}.
		
		\bibitem{H2}
		R.~{Hartshorne}.
		\newblock {Stable reflexive sheaves}.
		\newblock {\em {Math. Ann.}}, 254:121--176, 1980.
		\newblock \href {http://dx.doi.org/10.1007/BF01467074}
		{\path{doi:10.1007/BF01467074}}.
		
		\bibitem{HH}
		R.~{Hartshorne} and A.~{Hirschowitz}.
		\newblock {Cohomology of a general instanton bundle}.
		\newblock {\em {Ann. Sci. \'Ec. Norm. Sup\'er. (4)}}, 15:365--390, 1982.
		
		\bibitem{HSdp}
		R.~Hartshorne and E.~Schlesinger.
		\newblock Curves in the double plane.
		\newblock {\em Comm. Algebra}, 28(12):5655--5676, 2000.
		\newblock Special issue in honor of Robin Hartshorne.
		\newblock \href {http://dx.doi.org/10.1080/00927870008827180}
		{\path{doi:10.1080/00927870008827180}}.
		
		\bibitem{HSols}
		R.~Hartshorne and I.~Sols.
		\newblock Stable rank {$2$} vector bundles on {${\bf P}^{3}$} with
		{$c_{1}=-1,$} {$c_{2}=2$}.
		\newblock {\em J. Reine Angew. Math.}, 325:145--152, 1981.
		
		\bibitem{HL}
		D.~Huybrechts and M.~Lehn.
		\newblock {\em The geometry of moduli spaces of sheaves}.
		\newblock Cambridge Mathematical Library. Cambridge University Press,
		Cambridge, second edition, 2010.
		\newblock \href {https://doi.org/10.1017/CBO9780511711985}
		{\path{doi:10.1017/CBO9780511711985}}.
		
		\bibitem{Jou}
		J.~P. Jouanolou.
		\newblock {\em \'{E}quations de {P}faff alg\'{e}briques}, volume 708 of {\em
			Lecture Notes in Mathematics}.
		\newblock Springer, Berlin, 1979.
		
		\bibitem{DesP}
		M.~{Martin-Deschamps} and D.~{Perrin}.
		\newblock {Le sch\'ema de Hilbert des courbes gauches localement Cohen-Macaulay
			n'est (presque) jamais r\'eduit}.
		\newblock {\em {Ann. Sci. \'Ec. Norm. Sup\'er. (4)}}, 29(6):757--785, 1996.
		
		\bibitem{N}
		S.~Nollet.
		\newblock The {H}ilbert schemes of degree three curves.
		\newblock {\em Ann. Sci. \'{E}cole Norm. Sup. (4)}, 30(3):367--384, 1997.
		\newblock \href {http://dx.doi.org/10.1016/S0012-9593(97)89925-9}
		{\path{doi:10.1016/S0012-9593(97)89925-9}}.
		
		\bibitem{NS:deg4}
		S.~Nollet and E.~Schlesinger.
		\newblock Hilbert schemes of degree four curves.
		\newblock {\em Compositio Math.}, 139(2):169--196, 2003.
		\newblock \href {http://dx.doi.org/10.1023/B:COMP.0000005083.20724.cb}
		{\path{doi:10.1023/B:COMP.0000005083.20724.cb}}.
		
		\bibitem{OSS}
		C.~{Okonek}, M.~{Schneider}, and H.~{Spindler}.
		\newblock {\em {Vector bundles on complex projective spaces. With an appendix
				by S. I. Gelfand. Corrected reprint of the 1988 edition}}.
		\newblock Basel: Birkh\"auser, corrected reprint of the 1988 edition edition,
		2011.
		
		\bibitem{O}
		G.~Ottaviani.
		\newblock {\em Variet{\`a} proiettive di codimensione piccola}.
		\newblock Ist. nazion. di alta matematica F. Severi. Aracne, 1995.
		\newblock URL: \url{http://web.math.unifi.it/users/ottaviani/codim/codim.pdf}.
		
		\bibitem{RP:MS}
		R.~Piene and M.~Schlessinger.
		\newblock On the {H}ilbert scheme compactification of the space of twisted
		cubics.
		\newblock {\em Amer. J. Math.}, 107(4):761--774, 1985.
		\newblock \href {http://dx.doi.org/10.2307/2374355}
		{\path{doi:10.2307/2374355}}.
		
	\end{thebibliography}

\end{document}